\documentclass[11pt,a4paper,reqno]{amsart}

\usepackage{amsmath}
\usepackage{a4wide}
\usepackage[english]{babel}
\usepackage{amsfonts}
\usepackage{amssymb}
\usepackage{dsfont}
\usepackage{mathrsfs}
\usepackage{graphicx}
\usepackage{fancyhdr}
\usepackage{multicol}
\usepackage{enumitem}
\usepackage{amsthm}
\usepackage{empheq}
\usepackage{cases}
\usepackage[all]{xy}
\usepackage{stmaryrd}
\usepackage[colorlinks, citecolor=blue, linkcolor=red]{hyperref}
\usepackage{mathrsfs}
\usepackage{tikz,tikz-cd}

\theoremstyle{plain}

\newtheorem{teo}{Theorem}[section]
\newtheorem{definition}{Definition}[section]
\newtheorem{lemma}[teo]{Lemma}
\newtheorem{prop}[teo]{Proposition}
\newtheorem{cor}[teo]{Corollary}
\newtheorem{ackn}{Acknowledgments\!}

\theoremstyle{remark}

\newtheorem{rem}[teo]{Remark}

\tikzset{
	subset/.style={
		draw=none,
		edge node={node [sloped, allow upside down, auto=false]{$\subset$}}},
	Subset/.style={
		draw=none,
		every to/.append style={
			edge node={node [sloped, allow upside down, auto=false]{$\subset$}}}
	}
}
\tikzset{
	labl/.style={anchor=south, rotate=90, inner sep=.50mm}
}

\newcommand{\erre}{\mathds{R}}

\newcommand{\cinf}{C^{\infty}(M)}

\newcommand{\weyl}{\operatorname{W}}

\newcommand{\ra}{\rightarrow}

\newcommand{\set}[1]{{\left\{#1\right\}}}               
\newcommand{\pa}[1]{{\left(#1\right)}}                  
\newcommand{\sq}[1]{{\left[#1\right]}}                  
\newcommand{\abs}[1]{{\left|#1\right|}}                 

\newcommand{\eps}{\varepsilon}                           

\newcommand{\ol}[1]{\overline{#1}}
\renewcommand{\hat}[1]{\widehat{#1}}
\renewcommand{\tilde}[1]{\widetilde{#1}}

\newcommand{\hess}{\operatorname{Hess}}
\newcommand{\p}{\varphi}
\newcommand{\hs}{\mathrm{Hess}}

\newcommand{\ric}{\mathrm{Ric}}

\numberwithin{equation}{section}
\allowdisplaybreaks
\title[]
{Some geometric properties of generalized $\p$-vacuum static spaces}

\author[L. Branca]{Letizia Branca}
\address[Letizia Branca]{Dipartimento di Matematica, Universit\`{a} degli Studi di Milano, Via Saldini 50, 20133 Italy.}
\email[]{letizia.branca@unimi.it}

\author[P. Mastrolia]{Paolo Mastrolia}
\address[Paolo Mastrolia]{Dipartimento di Matematica, Universit\`{a} degli Studi di Milano, Via Saldini 50, 20133 Italy.}
\email[]{paolo.mastrolia@unimi.it}

\author[M. Rigoli]{Marco Rigoli}
\address[Marco Rigoli]{Dipartimento di Matematica, Universit\`{a} degli Studi di Milano, Via Saldini 50, 20133 Italy.}
\email[]{marco.rigoli@unimi.it}

\date{\today}
\subjclass[2010]{}

\begin{document}
	\maketitle

	\begin{abstract}
		In this paper we study the geometry of generalized $\p$-vacuum static spaces, proving estimates for the $\p$-scalar curvature and for the first eigenvalue of the Jacobi operator, and also rigidity under various geometric assumptions; in particular, we prove a result related to the famous Cosmic no-hair conjecture of Boucher, Gibbons and Horowitz.
	\end{abstract}
\section{Introduction and statement of the results}
The aim of the present paper is to study the geometric properties of what we call a \emph{generalized} $\p$-\emph{vacuum static space}, that is, an $m$-dimensional Riemannian manifolds $(M^m, g)$, $m\geq 3$, with possibly non-empty boundary $\partial M$,  for which we have a smooth solution $u$ of the problem
\begin{equation}\label{Eq1_system}
  	\begin{cases}
		u\ric^\p-\hess(u)=\Lambda g,\\
		u\tau(\p)+d\p(\nabla u)=0,
	\end{cases}
\end{equation}
with $u>0$, $u\equiv 0$ on $\partial M$ in case the latter is non-empty and where $\Lambda$ is a smooth function on $M$, $\p$ is a smooth map from $(M^m, g)$ to a second Riemannian manifold $(N^n, h)$ of dimension $n$, $\ric^\p$ is the $\p$-\emph{Ricci tensor} (see below for the definition) and $\tau(\p)$ is the \emph{tension field} of the map $\p$, that is the trace, with respect to $g$, of the generalized second fundamental tensor $\nabla d\p$ of the map $\p$, which extends the notion of mean curvature vector field for isometric immersions (see \cite{EL}).
 We recall the definition of the $\p$-Ricci tensor:
\begin{equation}\label{Eq2_phiRicci}
  \ric^\p = \ric - \alpha \varphi^*h,
\end{equation}
where $\alpha$ is a non-null real coupling constant; this $(0, 2)$-tensor has been first introduced by B. List, \cite{List2008EvolutionOA}, in his study of the Ricci-harmonic maps flow.
We shall introduce other modified curvature tensors of a similar nature: they merge the Riemannian geometry of $M$ with that of the map $\p$. Clearly, in case $\p$ is constant, $\ric^\p$ coincides with its classical counterpart $\ric$ (and this holds also for all the remaining $\p$-curvatures that we will define later); the $\p$-\emph{scalar curvature} is obtained from \eqref{Eq2_phiRicci} by contraction in the metric $g$, thus
\begin{equation}\label{Eq3_phiScalar}
  S^\p = S - \alpha\abs{d\p}^2;
\end{equation}
note that $\frac{1}{2}\abs{d\p}^2$ is the \emph{energy density} of the map $\p$. For suitable choices of the function $\Lambda$, system $\eqref{Eq1_system}$ describes several important structures, both from the mathematical and physical point of view: for instance, in case $\p$ is constant and
\begin{equation}\label{Eq4_LambdaCh1}
  \Lambda = \frac{S}{m-1}u,
\end{equation}
system \eqref{Eq1_system} recovers the equation of \emph{vacuum static spaces}; for
\begin{equation}\label{Eq5_LambdaCh2}
  \Lambda = -\frac{S}{m-1}u - \frac{\sigma}{m-1},
\end{equation}
where $\sigma$ is a real constant, we obtain the $V$-\emph{static equation} (see Miao and Tam, \cite{MT}). Observe that, in the paper of Miao and Tam, the constant $\sigma$ is equal to $-1$. However, their result can be extended to the case of  $\sigma \in \mathbb{R}$; moreover, in this case $S$ is constant (see Theorem 7 of \cite{MT}). For
\begin{equation}\label{Eq6_LambdaCh3}
  \Lambda = \pa{\mu+p} \frac{1}{m-1}u,
\end{equation}
with $\mu$ and $p$ smooth functions on $M$ to be interpreted as \emph{energy density} and \emph{pression} of a perfect fluid, we recover the \emph{static perfect fluid equation} (see e.g., \cite{CDPR2023,HawkingEllis} and the references therein). In the aforementioned examples, $\p$ is constant; otherwise, $\p$ is to be considered as a non-linear field, possibly interacting with a potential $U : N \ra \mathbb{R}$ (but we shall not presently consider this latter case, for which we refer to \cite{BCMMR}).

Our first result gives a lower bound on the $\p$-scalar curvature $S^\p$: we recall that, for $f\in C^\infty(M)$, the $f$-\emph{Laplacian} operator is defined as $\Delta_f = \Delta - g\pa{\nabla f, \cdot}$; moreover, we define
\[
S^\p_* = \inf_M S^\p.
\]
\begin{teo}\label{Th_7}
  Let $(M, g)$ be a complete manifold of dimension $m\geq 2$, satisfying \eqref{Eq1_system}. Assume $\alpha>0$,
  \begin{equation}\label{Eq8_Deltalog}
    \Delta_{-2\log{u}}\pa{\frac{\Lambda}{u}}\leq 0
  \end{equation}
  and
  \begin{equation}\label{Eq9_lambda}
    \lambda = \inf_M \pa{\frac{\Lambda}{u}} > -\infty.
  \end{equation}
  Then:
  \begin{itemize}
    \item[i)] if $\lambda>0$, then $M$ is compact and $(m-1)\lambda \leq S^\p_* \leq m\lambda$;
    \item[ii)] if $\lambda=0$, then $S^\p_* = 0$;
    \item[iii)] if $\lambda<0$, then $S^\p_* \geq m\lambda$.
  \end{itemize}
\end{teo}

\begin{rem}Note that:
	\begin{itemize}
		\item[i)] conditions \eqref{Eq8_Deltalog} and \eqref{Eq9_lambda} are automatically satisfied for a vacuum static space, that is for $\Lambda$ given in \eqref{Eq4_LambdaCh1}, since in this case $S$ is constant.
		\item[ii)] For a static perfect fluid, that is, for the choice of $\Lambda$ in \eqref{Eq6_LambdaCh3}, \eqref{Eq8_Deltalog} amounts to the request
		\begin{equation}\label{Eq10_Delta_mu_p}
			\Delta\pa{\mu+p} + 2 g\pa{\nabla\log u, \nabla(\mu+p)} \leq 0,
		\end{equation}
		while \eqref{Eq9_lambda} becomes
		\begin{equation}\label{infSPFST}
			\inf(\mu+p)>-\infty.
		\end{equation}
		However, in this case, by the dynamic of the fluid we know that
		\begin{equation}\label{Eq11_nabla p}
			\nabla p = -\pa{\mu+p}\frac{\nabla u}{u},
		\end{equation}
		so that \eqref{Eq10_Delta_mu_p} becomes
		\begin{equation}\label{Eq12_Delta_mu_p}
			\Delta\pa{\mu+p} \leq 2 g\pa{\nabla p, \nabla\log(\mu+p)}
		\end{equation}
		whenever $\mu+p>0$, which, in turn, can be written as
		\begin{equation*}
			\Delta\log(\mu+p)\leq \frac{\pa{\abs{\nabla p}^2-\abs{\nabla \mu}^2}}{\pa{\mu+p}}
		\end{equation*}
		and \eqref{Eq9_lambda} is satisfied. What is interesting, in the latter case, is that neither \eqref{Eq8_Deltalog} nor \eqref{Eq9_lambda} depend on $u$. Similar considerations hold for $\p$ non-constant.
		\item[iii)] For the $V$-static equation, since $S$ is constant, \eqref{Eq8_Deltalog} becomes
		\begin{equation*}
			-\frac{\sigma}{m-1}\Delta_{-2\log u}\pa{u^{-1}}\leq 0,
		\end{equation*}
		which is equivalent to
		\[\sigma\Delta u\leq 0,\]
		while \eqref{Eq9_lambda} becomes
		\[\inf_M \pa{-\frac{\sigma}{u}}>-\infty,\]
		which is automatically satisfied for $\sigma<0$ (which is the case of Miao and Tam, see the discussion above).
		
	\end{itemize}
\end{rem}

Our second result is related to the famous \emph{Cosmic no-hair   conjecture} of Boucher, Gibbons and Horowitz (see \cite{BGH}), that can be rephrased in the following form:

``\emph{an} $m$-\emph{dimensional, compact, static manifold} $(M, g)$ \emph{with positive scalar curvature and connected boundary} $\partial M\neq \emptyset$ \emph{is isometric to a round hemisphere} $\mathbb{S}_+^m(c)$ \emph{with an appropriate radius} $c$''.
This conjecture has been confirmed under different further hypotheses, but disproved for $\operatorname{dim}(M)\geq 4$ (for more information, see e.g. \cite{Gall2002}, \cite{GHP}, \cite{CDPR2023}).
We prove the following
\begin{teo}\label{Th_13}
   Let $(M, g)$ be a compact, connected manifold of dimension $m\geq 2$ with boundary $\partial M \neq \emptyset$. Let $\p : (M, g) \ra (N, h)$ be a smooth map, $\alpha>0$, $\Lambda\in C^\infty(M)$ and suppose that $u$ is a $C^2$-solution of
   \begin{equation}\label{Eq14_systemB}
  	\begin{cases}
		u\ric^\p-\hess(u)=\Lambda g,\\
        u>0 \quad \text{on}\,\, \operatorname{int}(M), \\
		\partial M = u^{-1}\pa{\set{0}},
	\end{cases}
\end{equation}
with $\frac{\Lambda}{u} \in C^0(M)$. Assume the following conditions:
\begin{equation}\label{Eq15}
  uS^\p \geq (m-1)\Lambda,
\end{equation}
\begin{equation}\label{Eq16}
  \Delta_{\pa{2m-3}\log u}\pa{S^\p-m\frac{\Lambda}{u}} \leq 0 \quad \text{on} \,\, \operatorname{int}(M),
\end{equation}
\begin{equation}\label{Eq17}
  m(m-1)\abs{\nabla u}^2|_{\partial M} \leq \max_M\set{m\Lambda u - S^\p u^2}.
\end{equation}
Then
\begin{equation}\label{Eq18}
  m\frac{\Lambda}{u} - S^\p = c^2,
\end{equation}
with $c$ a positive constant, and $(M, g)$ is isometric to $\mathbb{S}_+^m(c^2)$.
\end{teo}
Some comments on the assumptions are in order. Suppose $\p$ is constant and $\Lambda = \frac{S}{m-1}u$, that is, we are dealing with a vacuum static space: then \eqref{Eq15} is automatically satisfied. The same is true for \eqref{Eq16}, since $S-m\frac{\Lambda}{u} = -\frac{S}{m-1}$, which is a constant. Note that \eqref{Eq17} becomes
\begin{equation}\label{Eq19}
  \abs{\nabla u}^2|_{\partial M} \leq \frac{S}{m(m-1)}\max_M\set{u^2},
\end{equation}
which is the usual gravitational constraint imposed on the boundary $\partial M$ (see e.g. \cite{BoHa}): it is meaningful since, in this case, $S>0$ and, obviously, $\max_M\set{u^2}>0$. This latter fact is not at all obvious for the general condition expressed in \eqref{Eq17}; however, to show that \eqref{Eq17} is meaningful we may reason as follows. Set
\begin{equation}\label{Eq20}
  \tilde{A}  = S^\p - m\frac{\Lambda}{u} \in C^0(M),
\end{equation}
so that, tracing the first equation in \eqref{Eq14_systemB}, we deduce
\begin{equation}\label{Eq21}
  \Delta u = \tilde{A}u.
\end{equation}
We claim that there exist $q\in \operatorname{int}(M)$ such that
\begin{equation}\label{Eq22}
  \tilde{A}(q)<0.
\end{equation}
Indeed, suppose the contrary and let $\tilde{A}\geq 0$ on $\operatorname{int}(M)$ and thus, by continuity, on $M$. Since $u$ attains its maximum on $\operatorname{int}(M)$, by \eqref{Eq21} and the maximum principle $u$ is constant, and therefore null, since $u\equiv 0$ on $\partial M$, and this is a contradiction. Hence \eqref{Eq22} holds: this is enough to guarantee that the right-hand side of \eqref{Eq17} is strictly positive.

Our third result is a rigidity theorem: first we recall that a manifold $(M,g)$ is \emph{harmonic-Einstein} (with respect to $\varphi:(M,g)\ra(N,h)$ and $\alpha\in\erre\setminus\set{0}$) if
\begin{equation}\label{Eq26}
	\begin{cases}
		\ric^{\varphi}=\frac{S^{\varphi}}{m}g,\\
		\tau(\varphi)=0,
	\end{cases}
\end{equation}
and  we require $S^{\varphi}$ constant in case $m=2$ (in case $m\geq3$, $S^{\varphi}$ is automatically constant, see \cite{ACR}).
\begin{teo}\label{Th_23}
   Let $(M, g)$ be a compact manifold of dimension $m\geq 2$ satisfying, for some $\Lambda\in C^\infty(M)$,
   \begin{equation}\label{Eq24_systemC}
  	\begin{cases}
		u\ric^\p-\hess(u)=\Lambda g,\\
		u\tau(\p)+d\p(\nabla u)=0,
	\end{cases}
\end{equation}
with
\begin{equation}\label{Eq25}
  u>0 \qquad \text{ on }\,\, M
\end{equation}
Assume $S^\p$ is constant; then $(M, g)$ is harmonic-Einstein.
Furthermore, if $u$ is non-constant, then $\p$ is constant and $(M, g)$ is isometric to a Euclidean sphere $\mathbb{S}^m(k)$ of constant sectional curvature
\[
k = \frac{S^\p}{m(m-1)}>0.
\]
\end{teo}
To introduce our last result, we need the following
\begin{definition}[Jacobi Operator]
	Let $(M,g)$ be a smooth Riemannian manifold with boundary $\partial M$ and let $\phi\in C^{\infty} (\partial M)$, then the \emph{Jacobi operator}, denoted by $J_g$, is defined as
	\begin{align}\label{J}
		J_g\phi=\Delta_g \phi+\pa{\ric_g(\nu,\nu)+\abs{\mathrm{II}}_g^2}\phi,
	\end{align}
	where $\Delta_g$ denotes the Laplace operator with respect to the metric $g$ on the boundary, $\nu$ is the outward unit normal and $\mathrm{II}$ denotes the second fundamental form of $\partial M$. The \emph{first eigenvalue} of $J_g$ is defined as
	\begin{align}\label{lambda 1}
		(\lambda_1)_g:=\inf_{\phi\neq0}\frac{-\int_{\partial M}J_g\phi\, dV_{g|_{\partial M}} }{\int_{\partial M}\phi^2dV_{g|_{\partial M}} }.
	\end{align}
\end{definition}
Moreover, given $\varphi:(M,g)\ra(N,h)$, where $(X,g)$ and $(N,h)$ are smooth Riemannian manifolds and $\varphi$ is a smooth function, we define the $\varphi$-\emph{Yamabe invariant} as
\begin{align}\label{phi Y intro}
	Y(M,[g])^{\varphi}=\inf_{\tilde{g}\in[g]}\mathrm{Vol}(M)^{-\frac{n-2}{n}}\int_M \tilde{S}^{\varphi}dV_{\tilde{g}},
\end{align}
where $[g]$ denotes the conformal class of $g$.
We say that $\tilde{g}\in[g]$ is a $\varphi$-\emph{Yamabe minimizer} if it achieves the infimum in \eqref{phi Y intro}; we also denote with $\tilde{\varphi}$  the map $\varphi:(M,\tilde{g})\ra(N,h)$.
\begin{teo}\label{t-vol est}
	Let $(M,g)$ be a compact Riemannian manifold of dimension $m\geq 3$ with non-empty connected boundary, and let $\varphi:(M,g)\ra(N,g_N)$ be a smooth map, where $(N,g_N)$ is a second Riemannian manifold of dimension $n$. Let us assume the validity of the equation
	\begin{align}\label{e-andradephi}
		u\ric^{\varphi}-\hs(u)=\Lambda g,
	\end{align}
	where $u\in\cinf$, $u>0$ on $M$ and $\partial M=u^{-1}(\set{0})$. Let $\tilde{g}\in[g]$ be such that $\tilde{g}|_{\partial M}$ is a $\varphi$-Yamabe minimizer on $\partial M$, $\Lambda\in\cinf$, $\frac{\Lambda}{u}\in C^0(M)$.
	Then
	\[(\lambda_1)_{\tilde{g}} \leq\frac{1}{2}\pa{(m-1)(m-2)\omega_{m-1}^{\frac{2}{m-1}}\mathrm{Vol}(\partial M)^{-\frac{2}{m-1}}-\tilde{S}^{\tilde{\varphi}}_{min}},\]
	where $\tilde{S}^{\tilde{\varphi}}$ is the $\tilde{\varphi}$-scalar curvature of the metric $\tilde{g}$ and $\omega_{m-1}$ is the volume of the unit $(m-1)$-sphere. Moreover, if equality holds, then $\tilde{S}^{\tilde{\varphi}}$ is constant on $\partial M$, $\tilde{\varphi}^a_m\equiv 0$, $\partial M$ is totally geodesic in $(M,\tilde{g})$ and $(\partial M,\tilde{g}|_{\partial M})$ is conformally equivalent to the standard sphere.
\end{teo}

\begin{rem}
	The technique used in the proof of Theorem \ref{t-vol est} is similar to the one used by Andrade in \cite{andrade}, where she considers Riemannian manifold $(M,g)$ with boundary satisfying the equation
	\[u\ric-\hess(u)=\Lambda g.\]
	In particular, she provides an estimate for the first Jacobi egenvalue (with respect to the metric $g$) in terms of $\mathrm{Vol}(\partial M)$ and the Yamabe invariant of the $(m-1)$-sphere. However, note that the assumptions in \cite{andrade} are stronger than ours, indeed the boundary is assumed to be Einstein with positive scalar curvature.
\end{rem}


\section{Proof of Theorem \ref{Th_7}}
In the rest of the paper we shall freely use the moving frame notation: indices will run in the ranges
\[
1\leq i, j, \ldots, \leq m, \qquad 1 \leq a, b, \ldots, \leq n,
\]
and $\set{\theta^i}$, $\set{\omega^a}$ will be, respectively, local orthonormal coframes on open sets $U\subseteq M$, $V\subseteq N$ such that $\varphi(U)\subseteq V$.
To prove Theorem \ref{Th_7} we shall need the following
\begin{prop}
  Let $(M, g)$ be a Riemannian manifold of dimension $m\geq 2$ satisfying system \eqref{Eq1_system}. Then, on the set $\set{x\in M : u(x)>0}$,
  \begin{equation}\label{Eq2.3}
    \frac{1}{2}\Delta_{-3\log u}S^\p = -\pa{S^\p-m\frac{\Lambda}{u}}\pa{S^\p-(m-1)\frac{\Lambda}{u}}+(m-1)\Delta_{-2\log u}\pa{\frac{\Lambda}{u}}.
  \end{equation}
\end{prop}
\begin{proof}
  We take the covariant derivative of the first equation in \eqref{Eq1_system}, obtaining
  \begin{equation}\label{Eq2.4}
    u_kR^\p_{ij} + uR^\p_{ij, k} - u_{ijk} = \Lambda_k\delta_{ij}.
  \end{equation}
  Reversing the roles of $j$ and $k$ in \eqref{Eq2.4} and subtracting from \eqref{Eq2.4} we deduce
  \begin{equation}\label{Eq2.5}
     u_kR^\p_{ij}- u_jR^\p_{ik} + u\pa{R^\p_{ij, k}-R^\p_{ik, j}} = u_{ijk}-u_{ikj} + \Lambda_k\delta_{ij}-\Lambda_j\delta_{ik}.
  \end{equation}
  From the standard Ricci commutation relations we know that
  \[
  u_{ikj} = u_{ijk} + u_tR_{tikj};
  \]
  thus, from \eqref{Eq2.5} we get
  \begin{equation}\label{Eq2.6}
    u\pa{R^\p_{ij, k}-R^\p_{ik, j}} = -u_tR_{tikj} + \Lambda_k\delta_{ij}-\Lambda_j\delta_{ik} -  u_kR^\p_{ij}+ u_jR^\p_{ik}.
  \end{equation}
  Contracting with respect to $i$ and $j$, using the definition of $\ric^\p$ and the $\p$-\emph{Schur's identity}, that is,
  \begin{equation}\label{phiS}
  R^\p_{ij, j} = \frac{1}{2} S^\p_i - \alpha\p^a_{tt}\p^a_i
  \end{equation}
  (see for instance \cite{ACR} for a proof), we obtain
  \begin{equation}\label{Eq2.7}
    \frac{1}{2}uS^\p_k = -\alpha \p^a_k\pa{u_t\p^a_t+u\p^a_{tt}} + (m-1)\Lambda_k - u_kS^\p,
  \end{equation}
  and therefore, using  the second equation in \eqref{Eq1_system}, we obtain
  \begin{equation}\label{Eq2.8}
    \frac{1}{2}uS^\p_k =  (m-1)\Lambda_k - u_kS^\p.
  \end{equation}
  Contracting \eqref{Eq2.8} with $\frac{u_k}{u^2}$ yields
  \begin{equation}\label{Eq2.9}
     \frac{1}{2}S^\p_k \frac{u_k}{u}=  (m-1)\frac{\Lambda_k u_k}{u^2} - S^\p\frac{\abs{\nabla u}^2}{u^2}.
  \end{equation}
  Now we compute the divergence of $\nabla S^\p$, using the relation
  \[
    \frac{1}{2}S^\p_k =  (m-1)\frac{\Lambda_k}{u} -S^\p \frac{u_k}{u},
  \]
  which is obtained from \eqref{Eq2.8} dividing by $u$. After some algebraic manipulations we get
  \begin{align}\label{Eq2.10}
    \frac{1}{2}\Delta S^\p &= \frac{m-1}{u}\pa{\Delta\Lambda-\frac{\Lambda}{u}\Delta u} + \pa{\frac{m-1}{u}\Lambda-S^\p}\frac{\Delta u}{u} + S^\p\frac{\abs{\nabla u}^2}{u^2} \\ \nonumber &-\frac{m-1}{u^2}g\pa{\nabla\Lambda, \nabla u}-\frac{1}{u}g(\nabla S^{\varphi},\nabla u).
  \end{align}
  We observe that, tracing the first equation in \eqref{Eq1_system}, we have
  \begin{equation}\label{Eq2.11}
    \frac{\Delta u}{u} = S^\p - m\frac{\Lambda}{u};
  \end{equation}
  moreover, an easy computation shows that
  \begin{equation}\label{Eq2.12}
    \Delta_{{-2}\log u}\pa{\frac{\Lambda}{u}} = \frac{1}{u}\pa{\Delta\Lambda - \frac{\Lambda}{u}\Delta u},
  \end{equation}
  so that, inserting \eqref{Eq2.9}, \eqref{Eq2.11} and \eqref{Eq2.12} into \eqref{Eq2.10} we get
  \begin{equation*}
    \frac{1}{2}\pa{\Delta S^\p + 3g\pa{\nabla S^\p, \frac{\nabla u}{u}}} = (m-1) \Delta_{\log u^{-2}}\pa{\frac{\Lambda}{u}}-\pa{S^\p-m\frac{\Lambda}{u}}\pa{S^\p-(m-1)\frac{\Lambda}{u}},
  \end{equation*}
that is, equation \eqref{Eq2.3}.
\end{proof}
We are now ready for the
\begin{proof}[Proof of Theorem \ref{Th_7}]
  We perform the change of variable
  \begin{equation}\label{Eq2.13}
    u = e^{-f}
  \end{equation}
  so that system \eqref{Eq1_system} becomes
  \begin{equation}\label{Eq2.14}
  	\begin{cases}
		\ric^\p+\hess(f)-df\otimes df=\Lambda e^f g,\\
		\tau(\p)=d\p(\nabla f);
	\end{cases}
\end{equation}
having set $v = -S^\p$, equation \eqref{Eq2.3} becomes
\[
\frac{1}{2}\Delta_{3f}v = \pa{v+m\Lambda e^f}\pa{v+(m-1)\Lambda e^f} - (m-1)\Delta_{2f}(\Lambda e^f),
\]
  so that, by assumption \eqref{Eq8_Deltalog},
  \begin{equation}\label{Eq2.15}
    \frac{1}{2}\Delta_{3f}v \geq \pa{v+m\lambda e^f}\pa{v+(m-1)\Lambda e^f}.
  \end{equation}
  From the first equation in \eqref{Eq2.14}, $\alpha>0$ and \eqref{Eq9_lambda} we have
  \begin{equation}\label{Eq2.16}
    \ric + \hess{f} - df\otimes df \geq \lambda g.
  \end{equation}
  We now consider the three possible cases.

\textbf{1)} ${\lambda>0}$. 
Completeness of $(M, g)$ implies, by Theorem 5 of Z. Qian \cite{Qian}, that $M$ is compact. It follows that there exists $x_0\in M$ such that $-v(x_0) = S^\p_*$; we now show that
\begin{equation}\label{Eq2.17}
  S^\p_* \geq (m-1)\lambda.
\end{equation}
By contradiction, suppose this is not the case, that is,
\[
S^\p_* < (m-1)\lambda
\]
or, equivalently,
\begin{equation}\label{Eq2.18}
  \Omega = \set{x\in M : v(x)>-(m-1)\lambda} \neq \emptyset.
\end{equation}
Since $(m-1)\lambda<m\lambda$ we also have
\[
v(x) > -m\lambda \qquad \text{on } \,\, \Omega.
\]
Using \eqref{Eq9_lambda} it follows that, on $\Omega$,
\[
v + (m-1)\Lambda e^f \geq v + (m-1)\lambda > 0
\]
and
\[
v + m\Lambda e^f \geq v + m\lambda>0.
\]
Then \eqref{Eq2.15} gives
\begin{equation}\label{Eq2.19}
  \frac{1}{2}\Delta_{3f}v \geq \pa{v+m\lambda}\pa{v+(m-1)\lambda}\quad \text{on }\,\, \Omega,
\end{equation}
but $x_0\in \Omega$ and is a maximum of $v$: it follows that
\[
0 \geq \pa{-S^\p_*+m\lambda}\pa{-S^\p_*+(m-1)\lambda}
\]
from which we immediately deduce the validity of the inequality in the conclusion i).

\textbf{2)} ${\lambda=0}$. We show that $S^\p_* \geq 0$. By contradiction suppose that this is not the case, so that
\[
\Omega = \set{x\in M : v(x)>0} \neq \emptyset.
\]
Using \eqref{Eq9_lambda} it follows that, on $\Omega$,
\[
v+ (m-1)\Lambda e^f \geq v > 0 \quad \text{and }\,\,\, v+ m\Lambda e^f \geq v > 0.
\]
Thus \eqref{Eq2.15} gives
\begin{equation}\label{Eq2.21}
  \frac{1}{2}\Delta_{3f}v \geq v^2\quad \text{on }\,\, \Omega;
\end{equation}
from \eqref{Eq2.16} we obtain
\begin{equation}\label{Eq2.22}
  \ric + \hess(f) \geq 0.
\end{equation}
By Proposition 8.11 of \cite{AMR} for $\lambda \leq 0$, from \eqref{Eq2.16} we deduce
\begin{equation}\label{Eq2.23}
  \liminf_{r\ra+\infty}\frac{1}{r^2} \log\operatorname{vol}_f(B_r) < +\infty,
\end{equation}
where $\operatorname{vol}_f(B_r) = \int_{B_r(o)}e^{-f}$ for some fixed origin $o\in M$. Now we first apply Theorem 4.2 of \cite{AMR} to deduce from \eqref{Eq2.21} that $v^* < +\infty$, and then Theorem 4.1, again from \cite{AMR}, to infer $v^*\leq 0$, so that $S^\p_*\geq 0$.

\textbf{3)} ${\lambda<0}$. We claim that $S^\p_*\geq \lambda$. Suppose the contrary: then
\[
\Omega = \set{x\in M : v(x)>-m\lambda} \neq \emptyset.
\]
Since $(m-1)\lambda > m\lambda$ we also have
\[
v(x)>-(m-1)\lambda \quad \text{on }\,\, \Omega;
\]
then \eqref{Eq2.15} gives
\begin{equation}\label{Eq2.24}
  \frac{1}{2}\Delta_{3f}v \geq \pa{v+m\lambda}\pa{v+(m-1)\lambda}\quad \text{on }\,\, \Omega.
\end{equation}
The validity of \eqref{Eq2.23} and the completeness of $(M, g)$ enable us to apply again Theorems 4.2 and 4.1 of \cite{AMR} to conclude.
\end{proof}

\section{Proof of Theorem \ref{Th_13}}

We begin with the following observation, whose proof is, by now, quite standard; the aim is to emphasize the role of the assumption $\frac{\Lambda}{u}\in C^0(M)$. Let $M = \operatorname{int}(M) \cup \partial M$.
\begin{lemma}\label{Lemma3.1}
  Let $(M, g)$ be a compact manifold with boundary $\partial M \neq \emptyset$ and let $u\in C^2(M)$ be a solution of
  \begin{equation}\label{Eq3.2}
    u\ric^\p - \hess(u) = \Lambda g
  \end{equation}
  on $M$ for some $\Lambda\in C^\infty(M)$ and $\p : (M, g) \ra (N, h)$ a smooth map. Suppose $u>0$ on $\operatorname{int}(M)$ and $\partial M = u^{-1}\pa{\set{0}}$; then $\abs{\nabla u}$ is a positive constant on each connected component of $\partial M$ and the inclusion $\imath : \partial M \hookrightarrow M$ is totally geodesic, provided
  \begin{equation}\label{Eq3.3}
    \frac{\Lambda}{u}\in C^0(M).
  \end{equation}
\end{lemma}
\begin{proof}
  For the sake of completeness, we include the proof here. Using \eqref{Eq3.2} we obtain
  \[
  \nabla \abs{\nabla u}^2 = 2u\sq{\ric^\p\pa{\nabla u, \cdot}^\sharp - \frac{\Lambda}{u}\nabla u}.
  \]
  Hence $u\equiv 0$ on $\partial M$ and \eqref{Eq3.3} show that $\nabla \abs{\nabla u}^2 \equiv 0$ on $\partial M$. We now show that $\abs{\nabla u}(p) \neq 0$ for $p\in \partial M$; towards this aim, let $\nu$ be the outward unit normal to $\partial M$ at $p$, $\eps>0$ sufficiently small and  $\gamma : [0, \eps) \ra M$ a unit speed geodesic such that $\gamma(0)=p$ and $\dot{\gamma}(0)=-\nu$. Define
  \[
  v(t) = \pa{u\circ\gamma}(t) \quad \text{on }\,\, [0, \eps).
  \]
  Using \eqref{Eq3.2} and the fact that $\gamma$ is a geodesic we have
  \[
  \begin{cases}
    v'' = \ric^\p\pa{\dot{\gamma}, \dot{\gamma}}v - \Lambda(\gamma),\\ v'(0) = g\pa{\nabla u(p), \dot{\gamma}(0)}, \\ v(0) = u(p) = 0.
  \end{cases}
  \]
  Therefore, if $\nabla u(p)=0$, $v'(0)=0$: in this case we have $v\equiv 0$ on $[0, \eps')$ for some $0<\eps'\leq \eps$. This is a contradiction, since $\gamma\pa{(0, \eps')} \subseteq \operatorname{int}(M)$ and $u>0$ on $\operatorname{int}(M)$. It follows that $\abs{\nabla u}$ is a positive constant on each connected component of $\partial M$; in particular,
  \[
  \nu = -\frac{\nabla u}{\abs{\nabla u}}
  \]
  is the outward unit normal on $\partial M$ and the second fundamental form in the direction of $\nu$ is given by
  \[
  \mathrm{II} = \frac{1}{\abs{\nabla u}}\hess(u)|_{T\partial M \times T\partial M} = \frac{u}{\abs{\nabla u}}\pa{\ric^\p-\frac{\Lambda}{u}g},
  \]
  so that, using assumption \eqref{Eq3.3}, we deduce that $\imath : \partial M \hookrightarrow M$ is totally geodesic.
\end{proof}

\begin{lemma}\label{Lemma3.4}
  Let $(M, g)$ be a manifold of dimension $m\geq 2$ and let $u\in C^2(M)$, $u>0$ on $\operatorname{int}(M)$, be a solution of
  \begin{equation}\label{Eq3.5}
    u\ric^\p-\hess(u)=\Lambda g
  \end{equation}
  for some $\Lambda\in C^\infty(M)$ and $\p : (M, g) \ra (N, h)$ a smooth map. Let $Z$ be the vector field defined on $\operatorname{int}(M)$ by
  \begin{equation}\label{Eq3.6_defZ}
    Z = \frac{1}{u}\nabla\sq{\abs{\nabla u}^2-\pa{\frac{S^\p}{m}-\frac{\Lambda}{u}}u^2}.
  \end{equation}
  Then the divergence of $Z$ is given by
  \begin{align}\label{Eq3.7_divZ}
    \operatorname{div}(Z) &= \frac{2}{u}\sq{\abs{\hess(u)}^2-\frac{\pa{\Delta u}^2}{m}} + \frac{2}{u}\pa{S^\p-(m-1)\frac{\Lambda}{u}}\abs{\nabla u}^2 + \frac{2\alpha}{u}\abs{d\p(\nabla u)}^2 \\ \nonumber &-\frac{u}{m}\sq{\Delta\pa{S^\p-m\frac{\Lambda}{u}}-\frac{2m-3}{u}g\pa{\nabla\pa{S^\p-m\frac{\Lambda}{u}}, \nabla u}}.
  \end{align}
\end{lemma}
\begin{proof}
  Tracing equation \eqref{Eq3.5} we get
  \begin{equation}\label{Eq3.8}
    \Delta u = \frac{A}{m-1}u,
  \end{equation}
  where we have set
  \begin{equation}\label{Eq3.9}
    A = (m-1)\pa{S^\p-m\frac{\Lambda}{u}}
  \end{equation}
  to simplify the writing. Using \eqref{Eq3.8} we infer
  \begin{equation}\label{Eq3.10}
    \frac{2}{u}g\pa{\nabla\Delta u, \nabla u} = \frac{2}{u^2}\Delta u \abs{\nabla u}^2 + 2g\pa{\nabla\pa{\frac{A}{m-1}}, \nabla u}.
  \end{equation}
  We rewrite the components of $Z$ in \eqref{Eq3.6_defZ} in the form
  \[
  Z_k = \frac{2}{u}u_{ik}u_i - \frac{2}{m(m-1)}Au_k -  \frac{u}{m(m-1)}A_k.
  \]
  Computing the divergence and recalling the definition of $\ric^\p$ we obtain
  \begin{align}\label{Eq3.11}
     \operatorname{div}(Z) &= \frac{2}{u}\sq{\abs{\hess(u)}^2-\frac{\pa{\Delta u}^2}{m}} -\frac{u}{m}\sq{\frac{3}{u}g\pa{\nabla\pa{\frac{A}{m-1}}, \nabla u} + \Delta\pa{\frac{A}{m-1}}} \\ \nonumber &+\frac{2}{u}\set{-\frac{1}{u}\hess(u)\pa{\nabla u, \nabla u} + \ric^\p\pa{\nabla u, \nabla u}  +   g\pa{\nabla\Delta u, \nabla u}+\alpha\abs{d\p(\nabla u)}^2}.
  \end{align}
  Using \eqref{Eq3.5}, \eqref{Eq3.8} and \eqref{Eq3.10} into \eqref{Eq3.11} we deduce
  \begin{align}\label{Eq3.12}
     \operatorname{div}(Z) &= \frac{2}{u}\sq{\abs{\hess(u)}^2-\frac{\pa{\Delta u}^2}{m}} +\frac{2}{u}\pa{S^\p-(m-1)\frac{\Lambda}{u}}\abs{\nabla u}^2+\frac{2}{u}\alpha\abs{d\p(\nabla u)}^2 \\ \nonumber &-\frac{u}{m}\sq{\Delta\pa{\frac{A}{m-1}} - (2m-3)g\pa{\nabla\pa{\frac{A}{m-1}}, \frac{\nabla u}{u}}}.
  \end{align}
  Inserting \eqref{Eq3.9} we obtain\eqref{Eq3.7_divZ}.
\end{proof}

We are now ready for the

\begin{proof}[Proof of Theorem \ref{Th_13}]
  We set
  \begin{equation}\label{Eq3.13}
    v = \abs{\nabla u}^2 - \frac{1}{m(m-1)}Au^2,
  \end{equation}
  where $A$ has been defined in \eqref{Eq3.9}. We now show that $v$ is constant on $M$: towards this aim, since $M$ is compact, by Lemma \ref{Lemma3.1} we have that there exists a constant $c>0$ such that
  \[
  \abs{\nabla u}^2 \geq c^2
  \]
  on each connected component of $\partial M$. We can thus fix $\delta>0$ sufficiently small such that, for each $0<\eps\leq \delta$ and
  \[
  M_\eps = \set{x\in M : u(x)>\eps}
  \]
  $\frac{1}{u}$ is positive and bounded on $M_\eps$ and $\partial M_\eps$ is a $C^1$-hypersurface. Under assumptions \eqref{Eq15} and \eqref{Eq16} of the theorem, together with $\alpha>0$, equation \eqref{Eq3.7_divZ} of Lemma \ref{Lemma3.4} gives the validity of
  \begin{equation}\label{Eq3.14}
    \Delta v - \frac{1}{u} g\pa{\nabla u, \nabla v} \geq 0 \quad \text{on }\,\, M_\eps.
  \end{equation}
  Hence, by the maximum principle,
  \begin{equation}\label{Eq3.15}
    \max_{M_\eps} v = \max_{\partial M_\eps} v \quad \text{for }\,\, 0<\eps \leq \delta.
  \end{equation}
  Since, for $0<\eps \leq \delta$, $M_\delta \subseteq M_\eps$, from \eqref{Eq3.15} we deduce
   \begin{equation}\label{Eq3.16}
    \max_{\partial M_\delta} v \leq \max_{\partial M_\eps} v.
  \end{equation}
  Now
  \begin{equation}\label{Eq3.17}
    \lim_{\eps \downarrow 0^+}\max_{\partial M_\eps} v \leq \max \abs{\nabla u}^2|_{\partial M},
  \end{equation}
  and therefore
  \begin{equation}\label{Eq3.18}
    \max_{\partial M_\delta} v \leq \max \abs{\nabla u}^2|_{\partial M}.
  \end{equation}
  Using the previous relation and assumption \eqref{Eq17} we infer
  \begin{equation}\label{Eq3.19}
    \max_{\partial M_\delta} v \leq \max_M\set{-\frac{1}{m(m-1)}Au^2}.
  \end{equation}
  We let
  \[
  K = \set{p\in M : -\frac{1}{m(m-1)}A u^2(p) = \max_M\set{-\frac{1}{m(m-1)}A u^2}}.
  \]
  $K$ is closed, and therefore compact (since $M$ is compact), and for $p\in K$ we have $u(p)>0$. It follows that
  \begin{equation}\label{Eq3.20}
    \min_K u = 2\eta
  \end{equation}
  for some $\eta>0$. Then, choosing $\delta \leq \eta$,
  \[
  K \subseteq M_\delta \setminus \partial M_\delta,
  \]
  and using \eqref{Eq3.19}, for each $p \in K$ we have
  \begin{align*}
    \max_K\set{-\frac{1}{m(m-1)}A u^2} &= -\frac{1}{m(m-1)}A u^2(p) \leq v(p) \leq \max_{M_\delta}v \leq \max_M\set{-\frac{1}{m(m-1)}A u^2}.
  \end{align*}
  Hence, for each $p\in K$,
  \[
  v(p) = \max_{M_\delta}\set{-\frac{1}{m(m-1)}A u^2},
  \]
  that is, $v$ assumes its absolute maximum at the interior point $p$ of $M_\delta$. Using \eqref{Eq3.14} and the maximum principle we deduce that
  \begin{equation}\label{Eq3.21}
    v \equiv \max_M\set{-\frac{1}{m(m-1)}A u^2}
  \end{equation}
  on the connected components $C_\delta$ of $M_\delta$ such that $C_\delta \cap K \neq \emptyset$. Letting $\delta \downarrow 0^+$ we infer that \eqref{Eq3.21} holds on $M$ and $v$ is constant; indeed, let $p, q\in \operatorname{int}(M)$: then, by connectedness, there exists a path $\gamma : [0, 1] \ra \operatorname{int}(M)$ such that $\gamma(0) = p$, $\gamma(1)=q$. Since $\gamma\pa{[0, 1]}$ is compact, then $\inf_{\gamma\pa{[0, 1]}} u > 0$; it follows that $\gamma\pa{[0, 1]} \subset M_\delta$ for $0 < \delta \ll 1$. Hence $q$ and $p$ are in the same connected component of $p$. Now choose $p\in K$; since $v$ is constant, the vector field $Z$ defined in \eqref{Eq3.6_defZ} is identically null and \eqref{Eq3.12} gives
  \begin{equation}\label{Eq3.22}
    \hess(u) = \frac{\Delta u}{m}g
  \end{equation}
  and
  \begin{equation}\label{Eq3.23}
    \pa{S^\p-(m-1)\frac{\Lambda}{u}}\abs{\nabla u}^2 = 0
  \end{equation}
  on $\operatorname{int}(M)$ and, by continuity, on $M$ itself. From \eqref{Eq3.8} and \eqref{Eq3.9} we have
  \[
  \Delta u = \pa{S^\p-m\frac{\Lambda}{u}}u = uS^\p-m\Lambda,
  \]
  so that inserting into \eqref{Eq3.22} yields
  \begin{equation}\label{Eq3.24}
    \hess(u) = \pa{\frac{S^\p}{m}u-\Lambda}g \quad \text{on }\,\, M.
  \end{equation}
  Taking the gradient of the constant function $v$, from \eqref{Eq3.13} we obtain
  \begin{equation}\label{Eq3.25}
    0 = 2\hess(u)\pa{\nabla u, \cdot}^\sharp-\frac{u^2}{m(m-1)}\nabla A - \frac{2}{m(m-1)}Au\nabla u.
  \end{equation}
  Thus, using \eqref{Eq3.24} into \eqref{Eq3.25} and \eqref{Eq3.9} we infer
  \[
  \nabla\pa{m\frac{\Lambda}{u}-S^\p} \equiv 0 \quad \text{on }\,\, M.
  \]
  Connectedness of $M$ and \eqref{Eq22} imply
  \[
  m\frac{\Lambda}{u} - S^\p = C^2
  \]
  for some constant $C \neq 0$. We can thus rewrite \eqref{Eq3.24} as
  \[
  \hess(u) = -\frac{C^2}{m}ug,
  \]
  with $u$ non-constant. Furthermore, since $\frac{\Lambda}{u}\in  C^0(M)$, $\imath : \partial M \hookrightarrow M$ is totally geodesic, and therefore we can apply Lemma 3 of Reilly (\cite{R}) to deduce that $M$ is isometric to $S^m_+(C^2)$.
\end{proof}

\begin{rem}
  Note that assumption \eqref{Eq15} and conclusion \eqref{Eq18} of Theorem \ref{Th_13} imply
  \[
  (0<) \,\,C^2 \leq \frac{\Lambda}{u}\quad \text{on }\,\,M.
  \]
\end{rem}
\section{Proof of Theorem \ref{Th_23}}

For the proof of Theorem \ref{Th_23} we need a series of formulas.
\begin{prop}\label{Pr6.1}
  Let $(M, g)$ be a manifold of dimension $m\geq 3$, $X$ a conformal vector field on $M$, $\alpha \in \mathbb{R}$, $\alpha\neq 0$ and $\p : (M, g) \ra (N, h)$ a smooth map. Set
  \begin{equation}\label{Eq6.2}
    w = \operatorname{div}(X).
  \end{equation}
  Then
  \begin{align}\label{Eq6.3}
    \hess(w) &=\frac{S^\p}{(m-1)(m-2)}wg + \frac{m}{2(m-1)(m-2)} g\pa{\nabla S^\p, X}g -\frac{m}{m-2}\mathcal{L}_X\ric^\p \\ \nonumber &-\frac{m}{m-2}\alpha\pa{\mathcal{L}_X\pa{\p^*h}-\frac{1}{2(m-1)}\operatorname{tr}\pa{\mathcal{L}_X\pa{\p^*h}}g},
  \end{align}
  where $\mathcal{L}_X$ denotes the Lie derivative in the direction of $X$ and $\operatorname{tr}$ is the trace with respect to the metric $g$. In particular,
  \begin{equation}\label{Eq6.4}
    \Delta w = -\frac{S^\p}{m-1}w - \frac{m}{2(m-1)}g\pa{\nabla S^\p, X} - \frac{m}{2(m-1)}\alpha \operatorname{tr}\pa{\mathcal{L}_X\pa{\p^*h}}.
  \end{equation}
\end{prop}

\begin{rem}
  Note that, when $\p$ is constant, \eqref{Eq6.3} and \eqref{Eq6.4} reduce, respectively, to
  \begin{equation}\label{Eq6.6}
    \hess(w) =\frac{S}{(m-1)(m-2)}wg + \frac{m}{2(m-1)(m-2)} g\pa{\nabla S, X}g -\frac{m}{m-2}\mathcal{L}_X\ric,
  \end{equation}
  \begin{equation}\label{Eq6.7}
     \Delta w = -\frac{S}{m-1}w - \frac{m}{2(m-1)}g\pa{\nabla S, X}.
  \end{equation}
  These two formulas are known (see for instance the book \cite{YB}) and we shall use them to prove \eqref{Eq6.3}. However, for the sake of completeness we shall provide a proof here.
\end{rem}
\begin{proof}[Proof (of Proposition \ref{Pr6.1})]
  To show the validity of \eqref{Eq6.3} we shall use \eqref{Eq6.6}; thus, we begin to prove the latter, and towards this aim we first establish \eqref{Eq6.7}. Since $X$ is conformal,
  \begin{equation}\label{Eq6.8}
    X_{ij}+X_{ji} = \frac{2}{m}w \delta_{ij};
  \end{equation}
  contracting \eqref{Eq6.8} with the Ricci tensor we infer
  \begin{equation}\label{Eq6.9}
    R_{ij}X_{ij}=\frac{w}{m}S.
  \end{equation}
  By the definition of $w$, and using the commutation relation
  \begin{equation}\label{Eq6.11}
    X_{ijk} - X_{ikj} = X_tR_{tijk},
  \end{equation}
  we compute
  \begin{align*}
    \Delta w &= \pa{X_{ii}}_{jj} = \pa{X_{iij}}_{j} = \pa{X_{iji} + X_kR_{kiij}}_j = X_{ijij} - \pa{X_iR_{ij}}_j \\ &= \pa{X_{ij}}_{ij} - R_{ij, j}X_i - R_{ij}X_{ij}.
  \end{align*}
  With the aid of \eqref{Eq6.8}, Schur's identity and \eqref{Eq6.9}, the last equality can be written in the form
  \begin{align}\label{Eq6.12}
    \Delta w &=  \pa{-X_{ji}+\frac{2}{m}w\delta_{ij}}_{ij} - \frac{1}{2}S_iX_i - \frac{S}{m}w \\ \nonumber &= - X_{jiij} + \frac{2}{m}\Delta w - \frac{1}{2}S_iX_i - \frac{S}{m}w.
  \end{align}
  Using the commutation relation (see e.g. \cite{CMbook})
  \[
  X_{ijkl} - X_{ijlk} = X_{tj}R_{tikl} + X_{it}R_{tjkl}
  \]
  we immediately obtain
  \begin{equation}\label{Eq6.13}
    X_{jiij} = X_{jiji},
  \end{equation}
  and inserting into \eqref{Eq6.12} we infer
  \[
  \Delta w = -\pa{X_{jij}}_i + \frac{2}{m}\Delta w - \frac{1}{2}S_iX_i - \frac{S}{m}w.
  \]
  Using once again \eqref{Eq6.11} and \eqref{Eq6.8}, Schur's identity and \eqref{Eq6.9} we get
  \begin{align*}
    \Delta w &= -\pa{X_{jji} + X_kR_{kjij}}_i + \frac{2}{m}\Delta w - \frac{1}{2}S_iX_i - \frac{S}{m}w \\ &= -X_{jjii} - \pa{R_{ki}X_k}_i + \frac{2}{m}\Delta w -\frac{1}{2}S_iX_i - \frac{S}{m}w \\ &= -\Delta w - R_{ki, i}X_k - R_{ki}X_{ki} + \frac{2}{m}\Delta w  -\frac{1}{2}S_iX_i - \frac{S}{m}w \\&= -\Delta w -S_iX_i  - \frac{2}{m}Sw + \frac{2}{m}\Delta w,
  \end{align*}
  that is, \eqref{Eq6.7}. We now compute $\hess(w)$.
  \begin{align*}
    w_{ij} &= X_{llij} = X_{lilj} + X_{tj}R_{tlli} + X_tR_{tlli, j} \\ &=X_{lijl} + X_{ti}R_{tllj}+ X_{lt}R_{tilj} - X_{tj}R_{ti}-X_tR_{ti, j} \\ &=\pa{-X_{il}+\frac{2}{m}X_{tt}\delta_{il}}_{jl} -X_{ti}R_{tj} - X_{tj}R_{ti}+ X_{lt}R_{tilj}-X_tR_{ti, j} \\ &=-X_{iljl} +\frac{2}{m}X_{ttji}-X_{ti}R_{tj} - X_{tj}R_{ti}+ X_{lt}R_{tilj}-X_tR_{ti, j}\\ &= -X_{ijll}-X_{tl}R_{tilj}-X_tR_{tilj, l}+\frac{2}{m}X_{ttji}-X_tR_{ti, j}-X_{ti}R_{tj} - X_{tj}R_{ti}+ X_{lt}R_{tilj} \\&=-X_{ijll}+\frac{2}{m}X_{ttji}-X_{tl}R_{tilj}+ X_{lt}R_{tilj} - X_{t}R_{ij,t} +X_tR_{tj, i}-X_tR_{ti, j}\\&-X_{ti}R_{tj}-X_{tj}R_{ti} \\&=-X_{ijll}+\frac{2}{m}X_{ttji}-X_{lt}R_{ltij}+X_t\pa{R_{tj, i}-R_{ti, j}}-\pa{X_tR_{ij, t}+X_{ti}R_{tj}+X_{tj}R_{it}}\\&=-X_{ijll}+\frac{2}{m}X_{ttji}-X_{lt}R_{ltij}+X_t\pa{R_{tj, i}-R_{ti, j}}-\pa{\mathcal{L}_X\ric}_{ij}.
  \end{align*}
  Since $w_{ij}=w_{ji}$, symmetrizing with respect to the indexes $i$ and $j$ we get
  \begin{align*}
    X_{llij} &= -\frac{1}{2}\pa{X_{ij}+X_{ji}}_{ll} + \frac{2}{m}X_{llij}-\pa{\mathcal{L}_X\ric}_{ij} \\&=-\frac{1}{m}X_{kkll}\delta_{ij}+ \frac{2}{m}X_{llij}-\pa{\mathcal{L}_X\ric}_{ij}.
  \end{align*}
  Hence, rearranging the terms and using \eqref{Eq6.7}, we obtain \eqref{Eq6.6}.

  Starting from \eqref{Eq6.6}, we now prove \eqref{Eq6.3}. Towards this aim, we observe that from $S = S^\p+\alpha\abs{d\p}^2$, differentiating and contracting with $X$ we have
  \begin{equation}\label{Eq6.14}
    S_tX_t = S^\p_tX_t + 2 \alpha\p^a_{kt}\p^a_kX_t.
  \end{equation}
  From the simmetry of $\nabla d\p$ we get
  \begin{equation}\label{Eq6.15}
    \p^a_{kt}X_t = \p^a_{tk}X_t = \pa{\p^a_tX_t}_k-\p^a_tX_{tk},
  \end{equation}
  and therefore
  \[
  2 \alpha\p^a_{kt}\p^a_kX_t = 2 \alpha\pa{\p^a_{t}X_t}_k\p^a_k-2\alpha\p^a_t\p^a_kX_{tk}.
  \]
  Thus, using \eqref{Eq6.8},
  \begin{equation}\label{Eq6.16}
    2 \alpha\p^a_{kt}\p^a_kX_t = 2\alpha\pa{\p^a_{t}X_t}_k\p^a_k-\frac{2\alpha}{m}w\abs{d\p}^2.
  \end{equation}
  Inserting \eqref{Eq6.16} into \eqref{Eq6.14} we infer
  \begin{equation}\label{Eq6.17}
    S_tX_t =  S^\p_tX_t + 2\alpha\pa{\p^a_{t}X_t}_k\p^a_k-\frac{2\alpha}{m}w\abs{d\p}^2.
  \end{equation}
  Now we analyze $\mathcal{L}_X\ric$: since
  \[
  R_{ij} = R^\p_{ij} + \alpha\p^a_i\p^a_j,
  \]
  we have
  \[
  R_{ij, t} = R^\p_{ij, t} + \alpha\pa{\p^a_{it}\p^a_j+\p^a_i\p^a_{jt}}.
  \]
  Hence, using \eqref{Eq6.15},
  \begin{align*}
    X_tR_{ij, t} + X_{ti}R_{tj} + X_{tj}R_{it} &=  X_tR^\p_{ij, t} + \alpha X_t\pa{\p^a_{it}\p^a_j+\p^a_i\p^a_{jt}} + X_{ti}R^\p_{tj} \\ &+\alpha X_{ti}\p^a_t\p^a_j + X_{tj}R^\p_{it}  +\alpha X_{tj}\p^a_i\p^a_t \\ &= \pa{\mathcal{L}_X\ric^\p}_{ij} + \alpha\pa{\p^a_tX_t}_i\p^a_j+\alpha\pa{\p^a_tX_t}_j\p^a_i,
  \end{align*}
  that is,
  \begin{equation}\label{Eq6.18}
    \pa{\mathcal{L}_X\ric}_{ij} = \pa{\mathcal{L}_X\ric^\p}_{ij} + \alpha\pa{\mathcal{L}_X(\p^*h)}_{ij}.
  \end{equation}
  Inserting \eqref{Eq6.14} and \eqref{Eq6.18} in the local form of \eqref{Eq6.6}, we find \eqref{Eq6.3}.
\end{proof}
As a consequence of Proposition \ref{Pr6.1} we have the following
\begin{cor}\label{Cor6.18.1}
  In the assumptions of Proposition \eqref{Pr6.1}, suppose
  \begin{equation}\label{Eq6.18.2}
    \mathcal{L}_X(\p^*h) = 0.
  \end{equation}
  Then
  \begin{eqnarray}
  \label{Eq6.22}  \hess(w) &=& \frac{S^\p}{(m-1)(m-2)}wg + \frac{m}{2(m-1)(m-2)} g\pa{\nabla S^\p, X} -\frac{m}{m-2}\mathcal{L}_X\ric^\p, \\
  \label{Eq6.23}  \Delta w &=& -\frac{S^\p}{m-1}w-\frac{m}{2(m-1)}g\pa{\nabla S^\p, X}.  \end{eqnarray}
\end{cor}
\begin{rem}
  Equation \eqref{Eq6.18.2} geometrically means that $\p^*h$ is invariant under the flow of $X$.
\end{rem}
Note that, for a conformal vector field $X$,
\begin{equation}\label{Eq6.24}
  \mathcal{L}_X\ric^\p = \mathcal{L}_X{\overset{\circ}{\ric^\p}} + \frac{1}{m}\pa{g\pa{\nabla S^\p, X}+\frac{2}{m}S^\p w}g,
\end{equation}
where $\overset{\circ}{\ric^\p}$ is the traceless $\p$-Ricci tensor, so that \eqref{Eq6.22} can be written as
\begin{equation}\label{Eq6.25}
  \hess(w) -w\ric^\p-\Delta wg =  \frac{1}{2}g\pa{\nabla S^\p, X}g -\frac{m}{m-2}\mathcal{L}_X{\overset{\circ}{\ric^\p}} -w {\overset{\circ}{\ric^\p}} .
\end{equation}
Furthermore, suppose that $X$ is conformal and it satisfies
\begin{equation}\label{Eq6.26}
  \nabla \pa{d\p(X)}=0;
\end{equation}
then it also satisfies \eqref{Eq6.18.2}. Indeed, we have
\begin{equation}\label{Eq6.18.3}
  \pa{\mathcal{L}_X(\p^*h)}_{ij} = \pa{\p^a_tX_t}_i\p^a_j + \pa{\p^a_tX_t}_j\p^a_i.
\end{equation}
From the identity
\[
\pa{\pa{\p^a_tX_t}_s}_s=0,
\]
conformality of $X$, the commutation relations
\[
\p^a_{ijk} = \p^a_{ikj} + R_{tijk}\p^a_t + {}^N\!R_{abcd}\p^b_i\p^c_j\p^d_k
\]
and
\[
X_{tst} = X_{tts}+X_iR_{is}
\]
we deduce
\begin{align*}
	0&=\pa{\pa{\p^a_tX_t}_s}_s\\
	&=\p^a_{tss}X_t+\p^a_{ts}X_{ts}+\p^a_{st}X_{st}+\p^a_tX_{tss}\\
	&=\p^a_{sts}X_t+\frac{2}{m}w\p^a_{tt}+\p^a_tX_{tts}\\
	&=X_t\pa{\p^a_{sst}+\p^a_kR_{kt}+{}^NR_{abcd}\p^b_s\p^c_t\p^d_s}+\frac{2}{m}w\p^a_{tt}+\p^a_t\pa{-X_{sts}+\frac{2}{m}w_t}\\
	&=X_t\p^a_{sst}+X_t\p^a_kR_{kt}+{}^NR_{abcd}\p^b_s\p^c_t\p^d_sX_t+\frac{2}{m}w\p^a_{tt}+\frac{2}{m}\p^a_tw_t-\p^a_t\pa{X_{sst}+X_kR_{kt}}\\
	&=X_t\p^a_{sst}+{}^NR_{abcd}\p^b_s\p^c_t\p^d_sX_t+\frac{2}{m}w\p^a_{tt}+\frac{2}{m}\p^a_tw_t-\p^a_tw_t\\
	&=X_t\p^a_{sst}+{}^NR_{abcd}\p^b_s\p^c_t\p^d_sX_t+\frac{2}{m}w\p^a_{tt}-\frac{m-2}{m}\p^a_tw_t.
\end{align*}
Therefore, we have
\begin{align*}
	\frac{m-2}{m}\pa{w\p^a_{tt}+w_t\p^a_t}&=X_t\p^a_{sst}+{}^NR_{abcd}\p^b_s\p^c_t\p^d_sX_t+\frac{2}{m}w\p^a_{tt}-\frac{m-2}{m}\p^a_tw_t\\
	&+\frac{m-2}{m}\pa{w\p^a_{tt}+w_t\p^a_t}\\
	&=w\p^a_{tt}+\p^a_{sst}X_t+{}^NR_{abcd}\p^b_s\p^c_t\p^d_sX_t
\end{align*}
that gives  the validity of the formula
\begin{equation}\label{Eq6.27}
  w\tau(\p) + d\p\pa{\nabla w} = \frac{m}{m-2}\sq{w\tau(\p)+g\pa{\nabla \tau(\p), X}+{}^N\!R_{abcd}\p^b_s\p^c_t\p^d_sX_tE_a},
\end{equation}
where $\set{E_a}$ is the frame dual to the local orthonormal coframe $\set{\omega^a}$ considered on $N$.

Now, the condition
\begin{equation}\label{Eq6.28}
  X \in \operatorname{Ker}\pa{d\p}
\end{equation}
implies \eqref{Eq6.26}, and in this case \eqref{Eq6.27} simplifies to
\begin{equation}\label{Eq6.29}
   w\tau(\p) + d\p\pa{\nabla w} = \frac{m}{m-2}\sq{w\tau(\p)+g\pa{\nabla \tau(\p), X}}.
\end{equation}
We now state the following consequence of equations \eqref{Eq6.25} and \eqref{Eq6.29}.
\begin{prop}
  Let $(M, g)$ be a harmonic-Einstein manifold of dimension $m\geq 3$, with respect to $\p : (M, g) \ra (N, h)$ and $\alpha\in \mathbb{R}$, $\alpha\neq 0$. Let $X$ be  a conformal vector field on $M$ such that $X\in \operatorname{Ker}\pa{d\p}$. Then $w = \operatorname{div}X$ solves the system
  \begin{equation}\label{Eq6.31}
    \begin{cases}
      \hess(w) - \Delta wg -w\ric^\p = 0, \\ w\tau(\p) + d\p(\nabla w) = 0.
    \end{cases}
  \end{equation}
\end{prop}
\begin{rem}
  Of course, $w\not\equiv 0$ if and only if $X$ is not a Killing field.
\end{rem}
\begin{proof}
  Simply observe that, since $(M, g)$ is harmonic-Einstein, then $\overset{\circ}{\ric^\p}\equiv 0$, $\tau(\p)\equiv 0$ and $S^\p$ is constant. Then the conclusion follows from equations \eqref{Eq6.25} and \eqref{Eq6.29}.
\end{proof}
\begin{rem}\label{Rem6.35}
  Note that, in the present case, the second equation in \eqref{Eq6.31} becomes $d\p(\nabla w)=0$, that is, $\nabla w \in \operatorname{Ker}\pa{d\p}$.
\end{rem}
Here is another consequence of Proposition \ref{Pr6.1}, extending a well-known result of Yano and Nagano \cite{YN}. The role of the Einstein condition in \cite{YN} is here replaced by the harmonic-Einstein request.
\begin{prop}\label{Prop6.32}
  Let $(M, g)$ be a complete manifold of dimension $m\geq 2$, which is harmonic-Einstein with respect to $\p : (M, g) \ra (N, h)$ and $\alpha\in \mathbb{R}$, $\alpha\neq 0$. If there exists a conformal, non-Killing vector field $X$ on $M$ satisfying
  \begin{equation}\label{Eq6.34}
    \nabla d\p(X) = 0,
  \end{equation}
  then one the following alternatives occurs:
  \begin{itemize}
    \item[i)] $(M, g)$ is $\p$-Ricci-flat;
    \item[ii)] $(M, g)$ is isometric to a Euclidean sphere $\mathbb{S}^m(k)$ of constant sectional curvature $k=\frac{S^\p}{m(m-1)}>0$;
    \item[iii)] $(M, g)$ is isometric to a Hyperbolic space $\mathbb{H}^m(k)$ of constant sectional curvature $k=\frac{S^\p}{m(m-1)}<0$.
  \end{itemize}
\end{prop}
\begin{proof}
  From \eqref{Eq6.34} we know that $\mathcal{L}_X(\p^*h)=0$, $S^\p$ is constant and $\overset{\circ}{\ric^\p}\equiv 0$, thus from \eqref{Eq6.25} we deduce
  \begin{equation}\label{Eq6.36}
    \hess(w) - \Delta wg-\frac{S^\p}{m}wg =0,
  \end{equation}
  with $w=\operatorname{div}X$. Tracing \eqref{Eq6.36} gives
  \begin{equation}\label{Eq6.37}
    \Delta w = - \frac{S^\p}{m-1}w;
  \end{equation}
  inserting into \eqref{Eq6.36} yields
  \begin{equation}\label{Eq6.38}
    \hess(w) + \frac{S^\p}{m(m-1)}wg =0.
  \end{equation}
  Suppose now that $(M, g)$ is not $\p$-Ricci flat: then $w$ cannot be constant, because otherwise $w=a\neq 0$, since $X$ is not Killing, and from \eqref{Eq6.37} $S^\p\equiv 0$, contradicting the fact that $(M, g)$ is not $\p$-Ricci flat. Thus $w$ is a non-constant solution of \eqref{Eq6.38} and either $S^\p>0$ or $S^\p<0$, corresponding, by a theorem of Kanai \cite{Kan}, to cases ii) and iii), respectively.
\end{proof}
\begin{rem}
  In particular, in cases ii) and iii), the constant sectional curvatures are given by $\frac{S^\p}{m(m-1)}$ and
  \[
  \ric = \frac{S^\p}{m}g.
  \]
  Thus, using the harmonic-Einstein condition, we deduce
  \[
  \frac{S^\p}{m}g = \ric^\p = \ric -\alpha\p^*h = \frac{S^\p}{m}g-\alpha\p^*h
  \]
  and since $\alpha\neq 0$
  \[
  \p^*h \equiv 0,
  \]
  so that $\p$ is constant.
\end{rem}
\begin{rem}
  In case $M$ is compact, only alternative ii) can occur. Indeed, integrating \eqref{Eq6.37} on $M$, we obtain
  \[
  \int_M \abs{\nabla w}^2 = \frac{S^\p}{m-1}\int_M w^2,
  \]
  showing, since $w\not\equiv 0$, that $S^\p\geq 0$. However, $S^\p =0$ implies, by \eqref{Eq6.37}, $\Delta w=0$ on $M$, so that $w$ is constant. Thus
  \[
  w\operatorname{Vol}(M) = \int_M \operatorname{div} X = 0,
  \]
  and then $w=0$, which is a contradiction. It follows that $S^\p>0$.
\end{rem}
We need now a further preliminary result.
\begin{lemma}\label{Le6.45}
  Let $(M, g)$ be a manifold of dimension $m\geq 2$, $\p : (M, g) \ra (N, h)$ be a smooth map and $X$ a vector field on $M$. Then
  \begin{align}\label{Eq6.46}
    \operatorname{div}\pa{\overset{\circ}{\ric^\p}\pa{X, \,}^\sharp} = \frac{(m-2)}{2m}g\pa{X, \nabla S^\p} + \frac{1}{2}\operatorname{tr}\pa{\overset{\circ}{\ric^\p}\circ\mathcal{L}_Xg} - \alpha h\pa{\tau(\p), d\p(X)}.
  \end{align}
\end{lemma}
\begin{proof}
  We recall the $\p$-Schur's identity (see equation \eqref{phiS})
  \[
  R^\p_{ij, j} = \frac{1}{2} S^\p_i - \alpha\p^a_{tt}\p^a_i
  \]
  and we compute
  \begin{align*}
    \pa{ \overset{\circ}{R^\p_{ij}}}_j & = \pa{R^\p_{ij, j}-\frac{S^\p_i}{m}}X_i + R^\p_{ij}X_{ij} -\frac{S^\p}{m}\delta_{ij}X_{ij} \\&=X_i\pa{\frac{1}{2}S^\p_i-\alpha\p^a_{tt}\p^a_i}-\frac{1}{m}S^\p_iX_i +\frac{1}{2}R^\p_{ij}\pa{X_{ij}+X_{ji}}-\frac{1}{2}\frac{S^{\varphi}}{m}\delta_{ij}\pa{X_{ij}+X_{ji}}\\&=\frac{m-2}{2m}g\pa{X, \nabla S^\p} -\alpha h\pa{\tau(\p), d\p(X)}+\frac{1}{2}\operatorname{tr}\pa{\overset{\circ}{\ric^\p}\circ\mathcal{L}_Xg},
  \end{align*}
  that is, \eqref{Eq6.46}.
\end{proof}
We are now ready for the
\begin{proof}[Proof (of Theorem \ref{Th_23})]
  We let $u$ be the solution of \eqref{Eq24_systemC} and we set $X=\nabla u$ in \eqref{Eq6.46} to obtain\
  \begin{equation}\label{Eq6.52}
   \operatorname{div}\pa{\overset{\circ}{\ric^\p}\pa{\nabla u, \,}^\sharp} = \frac{(m-2)}{2m}g\pa{\nabla u, \nabla S^\p} +\pa{\abs{\overset{\circ}{\ric^\p}}^2+\alpha\abs{\tau(\p)}^2}u.
  \end{equation}
  We integrate \eqref{Eq6.52} on $M$ and we use the constancy of $S^\p$ to infer
  \[
  \int_M\pa{\abs{\overset{\circ}{\ric^\p}}^2+\alpha\abs{\tau(\p)}^2}u=0.
  \]
  Since $\alpha>0$ and $u>0$ on $M$, from the above we obtain
  \[
  \begin{cases}
    \ric^\p = \frac{S^\p}{m}g, \\ \tau(\p)=0,
  \end{cases}
  \]
  that is, $(M, g)$ is harmonic-Einstein. Suppose now that $u$ is non-constant; from the first equation in \eqref{Eq24_systemC} and the (just proved) validity of \eqref{Eq26} we have
  \[
  \hess(u) = \pa{\frac{S^\p}{m}-\frac{\Lambda}{u}}ug,
  \]
  so that $\nabla u$ is a conformal vector field on $M$. Furthermore,
  \[
  d\p(\nabla u) = -\tau(\p)u \equiv 0,
  \]
  and $\nabla u$ is non-Killing because otherwise, from \eqref{Eq26}, $\hess(u)\equiv 0$ and, by compactness of $M$, $u$ would be constant, a contradiction. From Proposition \ref{Prop6.32} and Remark \ref{Rem6.35} applied to $X=\nabla u$, we have the validity of conclusion ii) of Proposition \ref{Prop6.32}, and hence the theorem.
\end{proof}
\section{Proof of Theorem \ref{t-vol est}}\label{sec t-vol est}
This section is dedicated to the proof of Theorem \ref{t-vol est}, which provides an estimate for the first Jacobi eigenvalue of the  boundary, with respect to a suitable metric, of a Riemannian manifold satisfying the first equation of \eqref{Eq1_system}. The proof is based on the existence of a solution of the $\varphi$-Yamabe problem.\\

Let $(M,g)$ be a Riemannian manifold of dimension $m\geq3$; we will denote by $[g]$ the conformal class of the metric $g$, i.e.,
\[[g]:=\set{\tilde{g}=e^{2f}g,\, f\in\cinf}.\]
Yamabe, Trudinger, Aubin, and Schoen (see \cite{LeeParker} and the references therein for more details) proved that on a closed Riemannian manifold $(M,g)$ it is always possible to find a metric $\tilde{g}\in[g]$ such that the scalar curvature of $\tilde{g}$ is constant. The problem of finding such a metric is known as the \emph{Yamabe problem}: its solution involves variational methods applied to the normalized total scalar curvature functional.\\
Under the conformal change $\tilde{g}=u^{\frac{4}{m-2}}g$, for $m\geq 3$, $u\in C^\infty(M)$ positive, the scalar curvature
transforms according to the well-known formula

\begin{equation}\label{eq_star}
S_{\tilde{g}}=u^{-\frac{m+2}{m-2}}\pa{\dfrac{-4 (m-1)}{m-2}
	\Delta_gu+S_gu}=:u^{-\frac{m+2}{m-2}}L_gu,
\end{equation}
where the operator $L_g$ is  called the \emph{conformal
	Laplacian}. Solving the Yamabe problem, for $m\geq 3$, is
equivalent to finding a positive solution of the equation
\[
L_gu=\lambda u^{\frac{m+2}{m-2}}, \, \mbox{ with }
\lambda\in\erre.
\]
Note that the above equation is
the Euler-Lagrange equation of the normalized
Einstein-Hilbert functional
\[
\mathfrak{S}(\tilde{g})=
\mathrm{Vol}_{\tilde{g}}(M)^{-\frac{m-2}{m}}\int_M
\tilde{S}\, dV_{\tilde{g}},
\]
where $\tilde{g}\in [g]$ and $\tilde{S}$ denotes the scalar curvature with respect to the metric $\tilde{g}$.\\
To prove Theorem \ref{t-vol est} we introduce a suitable modification of the Yamabe invariant, which we will need to obtain a metric on $\partial M$ with constant $\varphi$-scalar curvature.
We now fix some notation. Let $(M,g)$ be a compact Riemannian manifold of dimension $m\geq3$; the Yamabe invariant of $(M,g)$ is defined as
\begin{align*}
	Y(M,[g])=\inf_{\tilde{g}\in[g]}\mathrm{Vol}(M)^{-\frac{m-2}{m}}\int_M \tilde{S}dV_{\tilde{g}}.
\end{align*}
Let $\varphi:(M,g)\ra(N,h)$ be a smooth map, where  $(N,h)$ is a second Riemannian manifold; in analogy with the usual Yamabe invariant,  define the $\varphi$-\emph{Yamabe invariant} by setting
\begin{align}\label{Yamabe phi}
	Y(M,[g])^{\varphi}=\inf_{\tilde{g}\in[g]}\mathrm{Vol}(M)^{-\frac{m-2}{m}}\int_M\tilde{S}^{\tilde{\varphi}}dV_{\tilde{g}},
\end{align}
where $\tilde{\varphi}$ is the map $\varphi:(M,\tilde{g})\ra(N,h)$.
By definition,
\begin{align*}
	Y^{\varphi}(M,[g])&=\inf_{\tilde{g}\in[g]}\mathrm{Vol}(M)^{-\frac{m-2}{m}}\int_M \tilde{S}^{\tilde{\varphi}}dV_g\\
	&=\inf_{\tilde{g}\in[g]}\mathrm{Vol}(M)^{-\frac{m-2}{m}}\pa{\int_M \tilde{S}dV_{\tilde{g}}-\int_M\alpha\abs{d\tilde{\varphi}}_{\tilde{g}}^2dV_{\tilde{g}}}.
\end{align*}
Observe that, when $\alpha>0$, we have
\begin{align*}
	Y^{\varphi}(M,[g])\leq Y(M,[g])\leq Y(\mathbb{S}^m,g_{\mathbb{S}^m}),
\end{align*}
so that
\begin{align*}
	Y^{\varphi}(M,[g])\leq m(m-1)\omega_m^{\frac{2}{m}}.
\end{align*}
\begin{rem}
	We point out that, in analogy with the Yamabe problem, Gursky and LeBrun in \cite{Gursky, GurskyLeBrun2} introduced the following modification of the Yamabe invariant
	\[\widehat{Y}(M,[g]):=\inf_{u\in W^{1,2}(M)}\frac{\int_M u\, \mathcal{L}^t_g u\,dV_g}{\left(\int_M u^{2m/(m-2)}\,dV_g\right)^{(m-2)/2}},\]
	where
	\[\mathcal{L}^t_g=-\frac{4(m-1)}{m-2}\Delta_g+S_g + t \abs{\weyl_g}_g\]
	is a modification of the usual conformal Laplacian. Note that $\widehat{Y}(M,[g])$ is conformally invariant; moreover, making use of a regularization argument, they proved that $\widehat{Y}(M,[g])\leq 0$ and then the  problem of finding a conformal metric $\tilde{g}\in[g]$ which attains $\widehat{Y}(M,[g])$ can be solved adapting the techniques in \cite{LeeParker}.
\end{rem}
For $m\geq 3$, let $\tilde{g}=u^{\frac{4}{m-2}}g$, $u\in C^\infty(M)$ positive; setting $\tilde{\varphi}$ for the map $\varphi$, considered ad a map from $(M, \tilde{g})$ to $(N, g_N)$, using \eqref{eq_star} we have

\begin{equation*}
	\tilde{S}^{\tilde{\varphi}}=u^{-\frac{m+2}{m-2}}\pa{-\frac{4(m-1)}{m-2}\Delta_g u+S^{\varphi}u}.
\end{equation*}

We can rewrite the $\varphi$-{Yamabe invariant} $Y^{\varphi}(M,[g])$ in the form
\begin{align}\label{Y phi u}
	Y^{\varphi}(M,[g])=Y^{\varphi}(u)=\inf_{u\neq0,u\in W^{1,2}}\frac{\int_M\pa{\frac{-4(m-1)}{m-2}u\Delta_g u+S^{\varphi}u^2}dV_g}{||u||_{L^{\frac{2m}{m-2}}}^2}=\inf_{u\neq0,u\in W^{1,2}}\frac{\int_Mu\mathcal{L}_g^{\varphi}u\, dV_g}{||u||^2_{L^{\frac{2m}{m-2}}}},
\end{align}
where $\mathcal{L}_g^{\varphi}$ is the second order elliptic operator

\begin{align}\label{conf lap phi}
	\mathcal{L}^{\varphi}_g:=-\frac{4(m-1)}{m-2}\Delta_g+S^{\varphi};
\end{align}
As a direct consequence of the transformation law of the $\varphi$-scalar curvature we have the following
\begin{lemma}
	Let $m\geq 3$ and $\tilde{g}=u^{\frac{4}{m-2}}g$; then
	$\mathcal{L}^{\tilde{\varphi}}_{\tilde{g}}\psi=u^{-\frac{m+2}{m-2}}\mathcal{L}^{\varphi}_{g}(\psi u)$.
\end{lemma}
Adapting an argument due to Gursky (\cite{Gursky}), we obtain the next
\begin{lemma}
	Each conformal class of a compact manifold $(M^m,g)$, $m\geq 3$, admits a $C^{\infty}$ metric $\tilde{g}=u^{\frac{4}{m-2}}g$ with either $\tilde{S}^{\tilde{\varphi}}>0$, $\tilde{S}^{\tilde{\varphi}}<0$ or $\tilde{S}^{\tilde{\varphi}}\equiv 0$; moreover, this three possibilities are mutually exclusive.
\end{lemma}
\begin{proof}
	The proof is based on that of Proposition 3.2 of \cite{Gursky}. Let $\mu(g)$ be the principal eigenvalue of $\mathcal{L}_g^{\varphi}$, namely
	\[
	\mu(g):=\inf_{u\neq 0,u\in W^{1,2}}\frac{\int_M u\mathcal{L}^{\varphi}_gu\,dV_g}{ \abs{\abs{u}}_{2}^2}.
	\]
	Since $\mathcal{L}_g^{\varphi}$ is conformally covariant, the sign of $\mu(g)$ is invariant under conformal changes of $g$. Let $\psi$ denote the first eigenfunction of the eigenvalue $\mu(g)$: by the minimum principle  $\psi$ can be assumed positive. Then
	from the equation
	$$\mathcal{L}_g^{\varphi}\psi=\mu(g)\psi,$$
	we have
	$$
	\Delta \psi=-\frac{(m-2)}{4(m-1)}\pa{S^{\varphi}\psi+\mu(g)\psi}
	$$
	and since $S^{\varphi}$ and $\mu(g)$ are smooth, by a bootstrap argument we conclude that $\psi$ is a smooth function.
	Let $\tilde{g}=\psi^{\frac{4}{m-2}}g$: then
	\begin{align*}
		\tilde{S}^{\tilde{\varphi}}&=\psi^{-\frac{m+2}{m-2}}\mathcal{L}_g^{\varphi}\psi\\
		&=\mu(g)\abs{\abs{\psi}}_{2}^2\psi^{-\frac{4}{m-2}}.
	\end{align*}
	
	Therefore,  $\tilde{S}^{\tilde{\varphi}}$ has the same sign  of $\mu(g)$:  positive, negative or identically zero, and since the sign $\mu(g)$ is invariant under conformal changes of the metric, these three alternatives are mutually exclusive.
	
\end{proof}
Given $(M,g)$ a closed Riemannian manifold of dimension $m\geq3$, it is always possible to find a metric $\tilde{g}\in[g]$ such that $\tilde{S}^{\tilde{\varphi}}$ is constant.
To prove this, it is sufficient to guarantee
\[Y^{\varphi}(M,[g])< Y(\mathbb{S}^m,[g_{\mathbb{S}^m}])\]
and to apply the arguments in \cite{LeeParker} and conclude that in every conformal class there exists a metric with constant $\varphi$-scalar curvature.
However, for the sake of completeness we provide a proof, which is obtained adapting an argument of Lee and Parker (\cite{LeeParker}) to our case.
On a compact manifold the Yamabe problem can be solved using a variational approach, provided $Y(M,[g])<Y(\mathbb{S}^m,g_{\mathbb{S}^m})$.
We focus on Section 4 of \cite{LeeParker}. Note that, by construction, for $\alpha>0$
\[Y^{\varphi}(M,[g])\leq Y(M,[g]).\]
Hence, if
\[Y(M,[g])<Y(\mathbb{S}^m,[g_{\mathbb{S}^m}])\]
we have
\[Y(M,[g])^{\varphi}<Y(\mathbb{S}^m,[g_{\mathbb{S}^m}]).\]
First let us fix some notation, let
\[Q_{\tilde{\varphi}}(\tilde{g})=\frac{\int_M \tilde{S}^{\tilde{\varphi}}dV_{\tilde{g}}}{\pa{\int_MdV_{\tilde{g}}}^{\frac{2}{p}}},\]
where $p=\frac{2m}{m-2}$.
As for the classical Yamabe problem $Q_{\varphi}$ can be written as
\[Q_{\tilde{\varphi}}(\tilde{g})=Q_{\varphi}(u)=\frac{E(u)}{\abs{\abs{u}}^2_p},\]
where
\[E(u)=\int_M\pa{-\frac{4(m-1)}{m-2}\abs{\nabla u}^2+S^{\varphi}u^2}dV_g.\]
Therefore, $u$ is a critical point for the $\varphi$-Yamabe invariant if it satisfies
\begin{equation}\label{e-pdeyphi}
	\mathcal{L}^{\varphi}_gu=\lambda u^{p-1},
\end{equation}
with $\lambda=\frac{E(u)}{\abs{\abs{u}}^p_p}.$\\
Note that, as in the classical Yamabe problem, the inclusion
\begin{equation}\label{inclusion}
	W^{1,2}\subset L^p
\end{equation}
is not compact. Let $\set{u_i}$ be a sequence of smooth functions such that $Q_{\varphi}(u_i)\ra Y^{\varphi}(M,[g])$ and suppose (by homogeneity) that $\abs{\abs{u_i}}_p=1$, for any $i$. Then, by H\"{o}lder's inequality, $\set{u_i}$ is bounded in $W^{1,2}(M)$ and there exists a subsequence converging to a fucntion $u\in W^{1,2}(M)$. However, by \eqref{inclusion}, we cannot guarantee that the constraint $\abs{\abs{u_i}}=1$ is preserved by $u$; moreover $u$ may be identically zero.\\
Following Lee and Parker (\cite{LeeParker})we consider an associated subcritical equation. We define
\[Q^s_{\varphi}(u):=\frac{E(u)}{\abs{\abs{u}}^2_s}\]
for $2\leq s\leq p$ and let
\[Y_s^{\varphi}=\inf\set{Q^s_{\varphi}(u):\,u\in C^{\infty}(X)}.\]
Note that if $u$ is a minimizing function with $\abs{\abs{u}}_s=1$, then
\begin{equation}\label{e-subcritical}
	\mathcal{L}_g^{\varphi}u=Y_s^{\varphi}u^{s-1}.
\end{equation}

We now state the analogous of Proposition 4.2 of \cite{LeeParker} in our setting. We omit the proof since it works exactly as that of \cite{LeeParker}.
\begin{prop}\label{p-4.2}
	For $s\in\sq{2,p}$ there exists a smooth positive solution $u_s$ to the subcritical equation \eqref{e-subcritical}, for which $Q^s_{\varphi}(u)=Y^{\varphi}_s$ and $\abs{\abs{u_s}}_s=1$.
\end{prop}

\begin{lemma}\label{l-mon}
	If $\mathrm{Vol}(M)=1$, then $\abs{Y^{\varphi}_s}$ is non-increasing as function of $s\in[2,p]$; moreover, if $Y^{\varphi}(M,[g])\geq0$, then $\abs{Y^{\varphi}_s}$ is left-continuous.
\end{lemma}
\begin{proof}
	Let $u\in\cinf$, $u\neq 0$. By H\"{o}lder inequality if $s\leq s'$, then $\abs{\abs{u}}_s\leq\abs{\abs{u}}_{s'}$. Now
	\[Q^{s'}_{\varphi}(u)=\frac{\abs{\abs{u}}_{s}^2}{\abs{\abs{u}}_{s'}^2}Q^s_{\varphi}(u)\]
	and then if $s\leq s'$ we have
	$Q^{s'}_{\varphi}\leq Q^s_{\varphi}$. It follows that
	\begin{equation}\label{e-noninc}
		\abs{Y^{\varphi}_{s'}}\leq\abs{Y^{\varphi}_s}.
	\end{equation}
	Assume that $Y^{\varphi}_p(M,[g])=Y^{\varphi}_p\geq0$, by \eqref{e-noninc} we have that
	\[Y^{\varphi}_s\geq0\]
	for every $s\in[2,p]$. Take $s\in\sq{2,p}$. Given $\eps>0$, there exists $u\in C^{\infty}(M)$ such that
	\[Q^{s}_{\varphi}(u)<Y^{\varphi}_s+\eps.\]
	By continuity of $\abs{\abs{u}}_s$, for $s'\leq s$, sufficiently close to $s$ we have
	\[Y^{\varphi}_{s'}\leq Q(u)^{s'}_{\varphi}\leq Y^{\varphi}_s+2\eps\]
	and since $Y^{\varphi}_s$ is non-increasing, we deduce that $Y^{\varphi}_s$ is continuos from the left.
\end{proof}
Following the proof in \cite{LeeParker}, we can now prove that if $Y^{\varphi}(M,[g])<Y(\mathbb{S}^m,[g_{\mathbb{S}^m}])$ holds, than the $\varphi$-Yamabe problem can be solved.
\begin{prop}\label{p-4.4}
	Let $\set{u_s}$ be the collection of functions  given by Proposition \ref{p-4.2}. There are constants $s_0<p$, $r>p$ and $C>0$ such that $\abs{\abs{u_s}}_r\leq C$ for every $s\geq s_0$.
\end{prop}
\begin{proof}
	Let $\delta>0$. We multiply \eqref{e-pdeyphi} times $u_s^{1+2\delta}$ and we integrate
	\begin{align*}
		Y^{\varphi}_s\int_Mu_s^{s+2\delta}dV_g&=\int_M\pa{-\frac{4(m-1)}{m-2}u_{s}^{1+2\delta}\Delta u_s+S^{\varphi}u_s^{2+2\delta}}dV_g\\
		&=\int_M\pa{\frac{4(m-1)}{m-2}u^{2\delta}(1+2\delta)\abs{\nabla u_s}^2+S^{\varphi}u^{2+2\delta}}dV_g.
	\end{align*}
	Let $w=u^{1+\delta}_s$, then we have
	\[\frac{(1+2\delta)}{(1+\delta)^2}\int_X\frac{4(m-1)}{m-2}\abs{\nabla w}^2dV_g=\int_M\pa{Y^{\varphi}_sw^2u^{s-2}-S^{\varphi}w^2}dV_g.\]
	To conclude the proof we recall the following theorems (see \cite{LeeParker} for more details).
	\begin{teo}\label{t-sobcon}
		Let $X$ be a compact manifold. Let $\sigma_M$ be the optimal Sobolev constant. For every $\eps>0$ there exists $C_{\eps}$ such that for every $v\in C^{\infty}(M)$
		\begin{equation}\label{e-sobc}
			\abs{\abs{v}}_{p}^2\leq(1+\eps)\sigma_M\int_M\abs{\nabla v}^2dV_g+C_{\eps}\int_Mv^2dV_g.
		\end{equation}
	\end{teo}
	\begin{teo}\label{t-sobsph}
			The $m$-dimensional Sobolev constant is $\sigma_M=\frac{4(m-1)}{m-2}/\Lambda,$ where $\Lambda=Y(\mathbb{S}^m,[g_{\mathbb{S}^m}])$. Thus, the sharp Sobolev inequality on $\erre^m$ is
			\begin{equation}\label{e-sharpsob}
				\abs{\abs{v}}_p^2\leq\frac{4(m-1)}{\Lambda(m-2)}\int_{\erre^m}\abs{\nabla v}^2dV_g.
			\end{equation}
	\end{teo}
	Combining \eqref{e-sobc}, \eqref{e-sharpsob} we deduce
	\begin{align*}
		\abs{\abs{w}}_p^2&\leq (1+\eps)\frac{4(m-1)}{\Lambda (m-2)}\int_M\abs{\nabla w}^2dV_g+C_{\eps}\int_Mw^2dV_g\\
		&=(1+\eps)\frac{(1+\delta)^2}{1+2\delta}\frac{1}{\Lambda}\int_MY^{\varphi}_sw^2u_{s}^{s-2}dV_g-(1+\eps)\frac{(1+\delta)^2}{1+2\delta}\frac{4(m-1)}{\Lambda(m-2)}\int_XS^{\varphi}w^2dV_g\\&+C_{\eps}\int_Mw^2dV_g\\
		&\leq(1+\eps)\frac{(1+\delta)^2}{1+2\delta}\frac{1}{\Lambda}\int_MY^{\varphi}_sw^2u_{s}^{s-2}dV_g+C_{\eps}'\int_Mw^2dV_g\\
	\end{align*}	
	and using H\"{o}lder inequality
	\begin{align*}
		\abs{\abs{w}}_p^2&\leq (1+\eps)\frac{(1+\delta)^2}{1+2\delta}\frac{Y^{\varphi}_s}{\Lambda}\abs{\abs{w}}_{p}^2\abs{\abs{u_s}}^{s-2}_{\frac{(s-2)m}{2}}+C_{\eps}'\abs{\abs{w}}_{2}^2.
	\end{align*}
	If $Y^{\varphi}(M,[g])<0$, then
	\[\abs{\abs{w}}_p^2\leq C\abs{\abs{w}}^2_2.\]
	If $0\leq Y^{\varphi}(M,[g])<Y(\mathbb{S}^m,[g_{\mathbb{S}^m}])$, since $Y^{\varphi}_s$ is non-increasing and left-continuous, we infer the existence of $s_0<p$, sufficiently close to $p$ such that
	\[\frac{Y^{\varphi}_s}{\Lambda}\leq\frac{Y^{\varphi}_{s_0}}{\Lambda}<1,\]
	for $s\geq s_0$. Moreover, since $\frac{(s-2)m}{2}<s$, by H\"{o}lder's inequality we have
	\[\abs{\abs{u_s}}_{\frac{(s-2)m}{2}}\leq\abs{\abs{u_s}}_s=1,\]
	therefore
	\[\abs{\abs{w}}_p^2\leq C\abs{\abs{w}}^2_2.\]
	To conclude the proof observe that
	\[\abs{\abs{w}}_2=\abs{\abs{u_s}}_{2(1+\delta)}^{1+\delta}\leq\abs{\abs{u_s}}_s^{1+\delta}.\]
	Therefore, $\abs{\abs{w}}_p$ is bounded independently of $s$.
\end{proof}
\begin{prop}\label{p-yamabe}
	Let $\set{u_s}$ be the functions given in Proposition \ref{p-4.2}. As $s\ra p$, $u_s$ converges (up to a subsequence) to a positive function $u\in C^{\infty}(X)$, which satisfies
	\begin{equation*}
		Q(u)^p_{\varphi}=Y^{\varphi}(M,[g]),\quad\mathcal{L}^{\varphi}u=Y^{\varphi}(M,[g])u^{p-1},
	\end{equation*}
	hence the metric $\tilde{g}=u^{p-2}g$ has constant $\varphi$-scalar curvature.
\end{prop}
\begin{proof}
	By Proposition \ref{p-4.4}, the sequence $\set{u_s}$ is uniformly bounded in $L^r(M)$ and by Theorem 4.1 of \cite{LeeParker}, they are uniformaly bounded in $C^{2,\alpha}$ and by Arzela-Ascoli Theorem a subsequence converges in $C^2$ norm to a function $u\in C^2(X)$ such that
	\begin{equation*}
		\mathcal{L}^{\varphi}_gu=\lambda u^{p-1}\quad Q(u)=\lambda,
	\end{equation*}
	where $\lambda=\lim_{s\ra p}Y^{\varphi}_s$. If $Y(M,[g])\geq 0$, by Lemma \ref{l-mon} $Y^{\varphi}_s$ is left-continuous, so that $\lambda=Y(M,[g])$. If $Y(M,[g])<0$, since $\abs{Y^{\varphi}_s}$ is non-increasing, we have that $Y^{\varphi}_s$ is increasing. Thus, $\lambda\leq Y(X,[g])$ and since $Y(M,[g])$ is the infimum of $Q^p_{\varphi}$, equality holds.
\end{proof}

Note that when $Y^{\varphi}(M,[g])=Y(\mathbb{S}^m,[g_{\mathbb{S}^m}])$, by definition of the $\varphi$-Yamabe invariant, we have that
\[Y^{\varphi}(M,[g])=Y^{\varphi}(M,[g])=Y(\mathbb{S}^m,[g_{\mathbb{S}^m}])\]
hence, by the classical results of Aubin and Schoen we have that $(M,[g])$ is conformally equivalent to the standard sphere. However, note that when equality holds it is not guaranteed that $\tilde{g}\in[g]$, such that $\tilde{S}^{\varphi}$ is constant coincide with the Yamabe metric.
\begin{rem}
	Note that the sign of the $\varphi$-Yamabe invariant is conformally invariant and it determines the sign of the $\varphi$-scalar curvature relative
	to the metric $\tilde{g}\in [g]$ that attains the minimum.
\end{rem}

Before proving Theorem \ref{t-vol est}, we racall that from Gauss equation, we have
\begin{align*}
	R_{ijkl}^{\partial M}&=R_{ijkl}-\mathrm{II}_{il}\mathrm{II}_{jk}+\mathrm{II}_{ik}\mathrm{II}_{jl};\\
	R_{ik}^{\partial M}&=R_{ik}^{\partial M}-R_{imkm}-\mathrm{II}_{il}\mathrm{II}_{lk}+\mathrm{II}_{ik}H(m-1);\\
	S^{\partial M}&=S-2R_{mm}+(m-1)(m-2)H^2,
\end{align*}
which imply
\begin{align}\label{S phi boundary}
	S^{\partial M,\varphi}&=S^{\varphi}+\alpha(\varphi^a_m)^2-2R_{mm}+(m-1)(m-2)H^2.
\end{align}
Moreover, we prove the following
\begin{lemma}\label{l-extension}
	Let $(M,g)$ be a compact Riemannian manifold with boundary $\partial M$ and let $f\in C^{\infty}(\partial M)$. Then, there exists $\ol{f}\in C^{\infty}(M)$ such that $\ol{f}|_{\partial M}=f$.
\end{lemma}
\begin{proof}
	Let $U_\eps$ be a tubular neighborhood of the boundary $\partial M$, namely, we consider
	\[U_\eps\cong\partial M\times[0,\eps)\]
	for $\eps>0$ sufficiently small.
	In particular we can identify a point $x\in U$ as $(y,t)$, where $y\in \partial M$ and $t\in[0,\eps)$.
	We define $F:U_\eps\ra\erre$, such that for every $x\in U_\eps$, $F(x)=F((y,t))=f(y)$, note that $F\in C^{\infty}(U_\eps)$ and $F|_{\partial M}\equiv f$.\\
	Let $\mathcal{U}=\set{U_i}_{i}$ be an open covering of the boundary such that $U_i\subset U_{\eps}$ for every $i$. Since $\partial M$ is compact, we can assume it is finite. We consider a partition of unit $\set{\psi_{i}}_i$ subordinate to $\mathcal{U}$; by construction $\mathrm{supp}(\psi_{i})\subseteq U_i$. Then, we define
	\begin{align*}
		\ol{f}(x)=	\begin{cases}
			\sum_i \psi_{i}(x)F(x)\quad\text{if }x\in U_{\eps};\\
			0\quad\text{if }x\in M\setminus U_{\eps}.
		\end{cases}
	\end{align*}
	By construction when $y\in\partial M$ $\ol{f}(y)=\sum_i \psi_{i}(y)f(y)=f(y)$ and $\ol{f}\in C^{\infty}(M)$.
\end{proof}
We are now ready to prove Theorem \ref{t-vol est}.
\begin{proof}[Proof of Theorem \ref{t-vol est}]
	By Proposition \ref{p-yamabe} we can always consider a metric $\tilde{g}_b\in[g|_{\partial M}]$ such that the $\varphi$-scalar curvature of the boundary associated to $\tilde{g}_b$ is constant.
	Since $\tilde{g}_b\in[g|_{\partial M}]$, there exists a function $f$ on $\partial M$ suh that $\tilde{g}_b=e^{2f}g$. By Lemma \ref{l-extension}, we can exten $f$ on $M$; in particular there exists $\ol{f}\in C^{\infty}(M)$ such that $\ol{f}|_{\partial M}=f$. Let us consider the metric $\tilde{g}:=e^{2\ol{f}}g$. By construction we have $\tilde{g}|_{\partial M}=\tilde{g}_b$ on the boundary and $\tilde{S}^{\tilde{\varphi},\partial M}$ is constant. Since $(M,g)$ is totally geodesic due to the validity of \eqref{e-andradephi}, the second fundamental form $\tilde{II}$ of $(M,\tilde{g})$ writes as
	$$
	\mathrm{\tilde{II}}_{ab}=e^{-\ol{f}}\pa{\mathrm{II}_{ab}+\frac{\partial \ol{f}}{\partial\nu}g_{ab}};
	$$
	and the mean curvature $\tilde{H}$ is given by
	$$\tilde{H}=\frac{\mathrm{tr}(\tilde{\mathrm{II}})}{m-1}=e^{-\ol{f}}\frac{\partial \ol{f}}{\partial\nu}.$$
	Therefore, by definition of $J_{\tilde{g}}$, we deduce
	\begin{align*}
		J_{\tilde{g}}\phi=\Delta_{\tilde{g}} \phi+\pa{\tilde{\ric}(\nu,\nu)+\frac{\tilde{H}^2}{(m-1)^2}}\phi,
	\end{align*}
	where $\tilde{\ric}$ denotes the Ricci tensor with respect to $\tilde{g}$.
	Moreover,
	\begin{align*}
		(\lambda_1)_{\tilde{g}}\int_{\partial M}\phi^2\,dV_{\tilde{g}_b}&\leq-\int_{\partial M}\phi J_{\tilde{g}}\phi\, dV_{\tilde{g}_b} \\
		&=\int_{\partial M}\pa{\abs{\nabla_{\tilde{g}} \phi}_{\tilde{g}}^2-\pa{\tilde{\ric}(\nu,\nu)+\frac{\tilde{H}^2}{(m-1)^2}}\phi^2}dV_{\tilde{g}_b}.
	\end{align*}
	Taking $\phi=1$ in the above equation and using \eqref{S phi boundary}, we get that when $\alpha>0$ the following inequalities hold
	\begin{align*}
		(\lambda_1)_{\tilde{g}}\mathrm{Vol}_{\tilde{g}_b}(\partial M)&\leq
		\frac{1}{2}\int_{\partial M}\tilde{S}^{\tilde{\varphi},\partial M}dV_{\tilde{g}_b}-\frac{1}{2}\int_{\partial M}\tilde{S}^{\tilde{\varphi}} dV_{\tilde{g}_b}-\int_{\partial M}\alpha(\tilde{\varphi}^a_m)^2dV_{\tilde{g}_b}-\int_{\partial M}\frac{\tilde{H}^2}{(m-1)^2}dV_{\tilde{g}_b}\\
		&\leq \frac{1}{2}\int_{\partial M}\tilde{S}^{\tilde{\varphi},\partial M}dV_{\tilde{g}_b}-\frac{1}{2}\int_{\partial M}\tilde{S}^{\tilde{\varphi}} dV_{\tilde{g}_b}\\
		&\leq \frac{1}{2}\int_{\partial M}\tilde{S}^{\tilde{\varphi},\partial M}dV_{\tilde{g}_b}-\frac{\tilde{S}^{\tilde{\varphi}}_{min}}{2}\mathrm{Vol}_{\tilde{g}_b}(\partial M)\\
		&\leq\frac{1}{2}Y^{\varphi}(\partial M,[g|_{\partial M}])\mathrm{Vol}_{\tilde{g}_b}(\partial M)^{\frac{m-3}{m-1}}-\frac{\tilde{S}^{\tilde{\varphi}}_{min}}{2}\mathrm{Vol}_{\tilde{g}_b}(\partial M),
	\end{align*}
	which implies
	\begin{align}\label{ineq lambda 1}
		(\lambda_1)_{\tilde{g}} \leq\frac{1}{2}\pa{(m-1)(m-2)\omega_{m-1}^{\frac{2}{m-1}}\mathrm{Vol}_{\tilde{g}_b}(\partial M)^{-\frac{2}{m-1}}-\tilde{S}^{\tilde{\varphi}}_{min}}.
	\end{align}
	If equality holds,
	\[Y^{\varphi}(\partial M,[g|_{\partial M}])=Y(\partial M,[g|_{\partial M}])=Y(\mathbb{S}^{m-1},[g_{\mathbb{S}^n}])\]
	and it follows directly by the classical results of Aubin and Schoen (see \cite{LeeParker} and the references therein) that $\partial M$ is conformally equivalent to a sphere. Furthermore, let $\hat{g}\in[g|_{\partial M}]$ be the metric which achieves the infimum of $Y(\partial M,[g|_{\partial M}])$: then
	\[\tilde{S}^{\tilde{\varphi},\partial M}=\hat{S}^{\partial M}\frac{\mathrm{Vol}_{\tilde{g}_b}(M)^{\frac{m-3}{m-1}}}{\mathrm{Vol}_{\hat{g}}(M)^{\frac{m-3}{m-1}}},\]
	where
	\[\frac{\mathrm{Vol}_{\tilde{g}_b}(M)^{\frac{m-3}{m-1}}}{\mathrm{Vol}_{\hat{g}}(M)^{\frac{m-3}{m-1}}}\]
	is constant.
	Moreover, if equality holds
	\[\int_{\partial M}\tilde{H}^2=0.\]
	which implies that the boundary is totally geodesic.
	Furthermore, we have that $\tilde{S}^{\tilde{\varphi}}=\tilde{S}^{\tilde{\varphi}}_{min}$ on $\partial M$, that is $\tilde{S}^{\tilde{\varphi}}$ is constant on $\partial M$. Moreover, $\tilde{\varphi}^a_m\equiv0$ on $\partial M$.
	
\end{proof}

\begin{rem}
	Observe that the proof of Theorem \eqref{t-vol est} is completely independent from the choice of $\Lambda$. For instance, when $\varphi$ is constant and $\Lambda$ assumes the values discussed in the introduction, we have the validity of the theorem for some well-known models (see the discussion in the introduction).
\end{rem}

\noindent{\bf Data availability statement}

\noindent Data sharing not applicable to this article as no datasets were generated or analysed during the current study.

	\ackn{The first and second authors are members of the	Gruppo Nazionale per le Strutture Algebriche, Geometriche e loro Applicazioni 	(GNSAGA) of INdAM (Istituto Nazionale di Alta Matematica). The authors are partially funded by 2022 PRIN project 20225J97H5 ``Differential Geometric Aspects of Manifolds via Global Analysis''.}

	\bibliographystyle{amsplain}
	\bibliography{bibliographyBMR}

\end{document}